\documentclass[final]{siamltex}
\usepackage{amsmath,amssymb,latexsym,enumerate,verbatim,amsfonts}
\usepackage{graphicx,xcolor,subfigure,multimedia}
\newcommand{\R}{{\mathbb{R}}}

\newcommand{\IN}{{\mathbb{N}}}

\numberwithin{equation}{section}
\newtheorem{remark}{Remark}[section]

\title{Linear convergence of inexact descent method and inexact proximal gradient algorithms for lower-order regularization problems}

\author{Yaohua Hu\thanks{College of Mathematics and Statistics, Shenzhen University, Shenzhen 518060, P. R. China (mayhhu@szu.edu.cn). This author's work was supported in part by the National Natural Science Foundation of China (11601343) and Natural Science Foundation of Guangdong (2016A030310038).}
\and Chong Li\thanks{School of Mathematical Sciences, Zhejiang University, Hangzhou 310027, P. R. China (cli@zju.edu.cn).
This author's work was supported in part by the National Natural Science Foundation of China (11571308).}
\and Kaiwen Meng\thanks{School of Economics and Management, Southwest Jiaotong University, Chengdu 610031, P. R. China (mkwfly@126.com).
This author's work was supported in part by the National Natural Science Foundation of China (11671329).}
\and Xiaoqi Yang\thanks{Department of Applied Mathematics, The Hong Kong Polytechnic University, Kowloon, Hong Kong (mayangxq@polyu.edu.hk). This author's work was supported in part by the Research Grants Council of Hong Kong (PolyU 152167/15E).}
}

\begin{document}
\maketitle

\begin{abstract}
The $\ell_p$ regularization problem with $0< p< 1$ has been widely studied for finding sparse solutions of linear inverse problems and gained successful applications in various mathematics and applied science fields. The proximal gradient algorithm is one of the most popular algorithms for solving the $\ell_p$ regularisation problem. In the present paper, we investigate the linear convergence issue of one inexact descent method and two inexact proximal gradient algorithms (PGA). For this purpose, an optimality condition theorem is explored to provide the equivalences among a local minimum, second-order optimality condition and second-order growth property of the $\ell_p$ regularization problem. By virtue of the second-order optimality condition and second-order growth property, we establish the linear convergence properties of the inexact descent method and inexact PGAs under some simple assumptions. Both linear convergence to a local minimal value and linear convergence to a local minimum are provided. Finally, the linear convergence results of the inexact numerical methods are extended to the infinite-dimensional Hilbert spaces.
\end{abstract}

\begin{keywords}
sparse optimization, nonconvex regularization, descent methods, proximal gradient algorithms, linear convergence.
\end{keywords}

\begin{AMS} Primary, 65K05, 65J22; Secondary, 90C26, 49M37 \end{AMS}

\section{Introduction}
The following linear inverse problem is at the core of many problems in various areas of mathematics and applied sciences:
finding $x\in \R^n$ such that
\[
Ax=b,
\]
where $A\in \R^{m\times n}$ and $b\in \R^m$ are known, and an unknown noise is included in $b$. If $m\ll n$, the above linear inverse problem is seriously ill-conditioned and has infinitely many solutions, and researchers are interested in finding solutions with certain structures, e.g., the sparsity structure. A popular technique for approaching a sparse solution of the linear inverse problem is to solve the $\ell_1$ regularization problem
\[
\min_{x\in \R^n}\; \|Ax-b\|^2+\lambda\|x\|_1,
\]
where $\|\cdot\|$ denotes the Euclidean norm, $\|x\|_1:=\sum_{i=1}^n |x_i|$ is a sparsity promoting norm, and $\lambda>0$ is a regularization parameter providing a tradeoff between accuracy and sparsity. In the past decade, the $\ell_1$ regularization problem has been extensively investigated (see, e.g., \cite{BeckTeboulle09,Combettes05,Daubechies04,Nesterov13,TZhang13,YZ11}) and gained successful applications in a wide range of fields, such as compressive sensing \cite{CandesTao05,Donoho06}, image science \cite{BeckTeboulle09,Elad-SRR}, systems biology \cite{SimonSGLasso2013,HuIP17} and machine learning \cite{BachETAL2012,Mairal2015}.

However, in recent years, it has been revealed by extensive empirical studies that the solutions obtained from the $\ell_1$ regularization may be much less sparse than the true sparse solution, and that the $\ell_1$ regularization cannot recover a signal or an image with the least measurements when applied to compressive sensing; see, e.g., \cite{Chartrand08,XuZB12,Zhang10}.
To overcome these drawbacks, the following $\ell_p$ regularization problem ($0< p< 1$) was introduced in \cite{Chartrand08,XuZB12} to improve the performance of sparsity recovery:
\begin{equation}\label{eq-lp}
\min_{x\in \R^n}\; \|Ax-b\|^2+\lambda\|x\|_p^p,
\end{equation}
where $\|x\|_p:=\left(\sum_{i=1}^n |x_i|^p\right)^{1/p}$ is the $\ell_p$ quasi-norm.
It was shown in \cite{Chartrand08} that the $\ell_p$ regularization requires a weaker restricted isometry property to guarantee perfect sparsity recovery and allows to obtain a more sparse solution from fewer linear measurements than that required by the $\ell_1$ regularization; and it was illustrated in \cite{HuJMLR17,XuZB12} that the $\ell_p$ regularization has a significantly stronger capability in obtaining a sparse solution than the $\ell_1$ regularization. 
Benefitting from these advantages, the $\ell_p$ regularization technique has been applied in many fields; see \cite{HuJMLR17,Marjanovic2014,Pant2014,QinHu2014} and references therein. It is worth noting that the $\ell_{p}$ regularization problem \eqref{eq-lp} is a variant of lower-order penalty problems, investigated in \cite{BurachikRubinov07,HuangYang03,LuoPangRalph96}, for a constrained optimization problem. The main advantage of the lower-order penalty functions over the classical $\ell_1$ penalty function in the context of constrained optimization is that they require weaker conditions to guarantee an exact penalization property and that their least exact penalty parameter is smaller.

Motivated by these significant advantages and successful applications of the $\ell_p$ regularization,
tremendous efforts have been devoted to
the study of optimization algorithms for the $\ell_p$ regularization problem. Many practical algorithms have been investigated for solving problem \eqref{eq-lp}, such as an interior-point potential reduction algorithm \cite{Ge10}, smoothing methods \cite{ChenMP2012,ChenXJ10}, splitting methods \cite{PTKMP15,TKSIOPT15} and iterative reweighted minimization methods \cite{LaiWang11,LuMP2014}. In particular, Xu et al. \cite{XuZB12} proposed an iterative half thresholding algorithm, which is efficient in signal recovery and image deconvolution.
In the present paper, we are particularly interested in the proximal gradient algorithm (in short, PGA) for solving problem \eqref{eq-lp}, which is reduced to the algorithm proposed in \cite{XuZB12} when $p=\frac12$.

{\scshape Algorithm PGA}.\,
Given an initial point $x^0\in \R^n$ and a sequence of stepsizes $\{v_k\}\subseteq \R_+$. For each $k\in \IN$, having $x^k$, we determine $x^{k+1}$ as follows:
\begin{eqnarray}
z^k&:=&x^k-2v_k A^\top  (Ax^k-b), \nonumber \\
x^{k+1}&\in&{\rm arg}\min_{x\in \R^n}\left\{\lambda \|x\|_p^p+\frac{1}{2v_k}\|x-z^k\|^2\right\}.\label{eq-PPA}
\end{eqnarray}
The PGA is one of the most widely studied first-order iterative algorithms for solving regularization problems, and a special case of several iterative methods (see \cite{AttouchBolte10,AttouchBolte13,BolteTeboulle13,TsengPaul09,RazaviyaynHong13}) for solving the composite minimization problem 
\begin{equation}\label{eq-CCO}
\min_{x\in \R^n}\; F(x):=H(x)+\Phi(x),
\end{equation}
where $H:\R^n\to \overline{\R}:=\R\cup \{+\infty\}$ is smooth and convex, and $\Phi:\R^n\to \overline{\R}$ is nonsmooth and possibly nonconvex.
The convergence properties of these iterative methods have been explored under the framework of so-call Kurdyka-{\L}ojasiewicz (in short, KL) theory. In particular, Attouch et al. \cite{AttouchBolte13} established the global convergence of abstract descent methods for minimizing a KL function $F:\R^n\to \overline{\R}$ (see \cite[Definition 2.4]{AttouchBolte13} for the definition of a KL function), in which the sequence $\{x_k\}$ satisfies the following hypotheses for two positive constants $\alpha$ and $\beta$:
\begin{enumerate}
  \item[(H1)] (\emph{Sufficient decrease condition}). For each $k\in \IN$,
  \begin{equation*}\label{asp-H1}
  F(x^{k+1})-F(x^k)\le -\alpha\|x^{k+1}-x^k\|^2;
  \end{equation*}
  \item[(H2)] (\emph{Relative error condition}). For each $k\in \IN$, there exists $w^{k+1} \in \partial F(x^{k+1})$ such that
  \begin{equation*}\label{asp-H2}
  \|w^{k+1}\| \le \beta\|x^{k+1}-x^k\|;
  \end{equation*}
  \item[(H3)] (\emph{Continuity condition})\footnote{This condition is satisfied automatically for the $\ell_p$ regularization problem \eqref{eq-lp}.}. There exist a subsequence $\{x^{k_j}\}$ and a point $x^*$ such that
  \[
  \lim_{j\to \infty}x^{k_j} \to x^*\quad \mbox{and} \quad \lim_{j\to \infty}F(x^{k_j}) \to F(x^*).
  \]
\end{enumerate}
The global convergence of Algorithm PGA follows from the established convergence results of \cite{AttouchBolte13}.


The study of convergence rates of optimization algorithms is an important issue of numerical optimization, and much attention has been paid to establish the convergence rates of relevant iterative algorithms for solving the structured optimization problem \eqref{eq-CCO}; see \cite{AttouchBolte10,Bolte2016,HuSIOPT16,PTKMP15,OchsiPiano2014,Tseng2010,TsengPaul09,WenBoTK2017,Yin13} and references therein.
For example, the linear convergence of the PGA for solving the classical $\ell_1$ (convex) regularization problem has been well investigated; see, e.g., \cite{BrediesLorenz08a,SZZhang2016,ZhangLuoJOSC2013,HuOpt17} and references therein.
Under the general framework of the KL (possibly nonconvex) functions, the linear convergence of several iterative algorithms for solving problem \eqref{eq-CCO}, including the PGA as a special case, have been established in \cite{AttouchBolte10,BolteTeboulle13,TsengPaul09,Yin13} under the assumption that the KL exponent of the objective function is $\frac12$. However, the KL exponent of the $\ell_q$ regularized function is still unknown, and thus, the linear convergence result in these references cannot be directly applied to the $\ell_q$ regularization problem \eqref{eq-lp}.
On the other hand, Zeng et al. \cite{zhenlinxu2015} obtained the linear convergence of the PGA 
for problem \eqref{eq-lp} with an upper bound on $p$, which may be less than $1$, and a lower bound on the stepsizes $\{v_k\}$, and
Hu et al. \cite{HuJMLR17} established the linear convergence of the PGA for the group-wised $\ell_p$ regularization problem under the assumption that the limiting point is a local minimum.

Another important issue is the practicability of the PGA for solving the $\ell_p$ regularization problem \eqref{eq-lp}. It is worth noting that the main computation of the PGA is the calculation of the proximity operator of the $\ell_p$ regularizer \eqref{eq-PPA}. The analytical solutions of the proximity operator of the $\ell_p$ regularizer \eqref{eq-PPA} when $p=1$ (resp. $\frac23$, $\frac12$, 0)  were provided in \cite{Daubechies04} (resp. \cite{XuZB23}, \cite{XuZB12}, \cite{Blumensath08}); see also \cite[Proposition 18]{HuJMLR17} for the group-wised $\ell_p$ regularizer. However, in the scenario of general $p$, the proximity operator of the $\ell_p$ regularizer may not have  an analytic solution (see \cite[Remark 21]{HuJMLR17}), and it could be computationally expensive to solve subproblem \eqref{eq-PPA} exactly at each iteration. Although some recent works showed impressive empirical performance of the inexact versions of the PGA that use an approximate proximity operator (see, e.g., \cite{HuJMLR17,Ma2011} and references therein), there is few theoretical analysis, to the best of our knowledge, on how the error in the calculation of the proximity operator affects the convergence rate of the inexact PGA for solving the $\ell_p$ regularization problem \eqref{eq-lp}. Two relevant papers on the linear convergence study of the inexact PGA should be mentioned: (a) Schmidt et al. \cite{SchmidtNIPS2011} proved the linear convergence of the inexact PGA for solving the convex composite problem \eqref{eq-CCO}, in which $H$ is strongly convex and $\Phi$ is convex;
(b) Frankel et al. \cite{Frankel-JOTA} provided a framework of establishing the linear convergence for descent methods satisfying (H1)-(H3), where (H2) is replaced by inexact form (H2$^\circ)$, see section 4. However, the convergence analysis in \cite{Frankel-JOTA} was based on the assumption that the KL exponent of $F$ is $\frac12$ and the inexact version would be not convenient to implement for applications; see the explanation in Remark \ref{rem-IPGA} below. Therefore, neither of the convergence analysis in \cite{Frankel-JOTA,SchmidtNIPS2011} can be applied to establish the linear convergence of the inexact PGA for solving the $\ell_q$ regularization problem. Thus, a clear analysis of the convergence rate of the inexact PGA is required to advance our understanding of its strength for solving the $\ell_p$ regularization problem \eqref{eq-lp}.

The aim of the present paper is to investigate the linear convergence issue of an inexact descent method and inexact PGAs for solving the $\ell_p$ regularization problem \eqref{eq-lp}. For this purpose, we first investigate an optimality condition theorem for the local minima of the $\ell_p$ regularization problem \eqref{eq-lp}, in which we establish the equivalences among a local minimum, second-order optimality condition and second-order growth property of the $\ell_p$ regularization problem \eqref{eq-lp}.
The established optimality conditions are not only of independent interest (which, in particular, improve the result in \cite{ChenXJ10}) in investigating the structure of local minima, but also provide a crucial tool for establishing the linear convergence of the inexact descent method and inexact PGAs for solving the $\ell_p$ regularization problem in sections 4 and 5.

We then consider a general framework of an inexact descent method, in which both (H1) and (H2) are relaxed to inexact forms (see (H1$^\circ$) and (H2$^\circ$) in section 4), for solving the $\ell_p$ regularization problem. Correspondingly, the solution sequence does not satisfy the descent property. This is an essential difference from the extensive studies in descent methods and the work of Frankel et al. \cite{Frankel-JOTA}. Under some mild assumptions on the limiting points and inexact terms, we establish the linear convergence of the inexact descent method by virtue of both second-order optimality condition and second-order growth property (see Theorem \ref{thm-LC}).

The convergence theorem for the inexact descent method further provides a useful tool for establishing the linear convergence of the inexact PGAs in section 5. Our convergence analysis deviates significantly from that of \cite{Frankel-JOTA} and relevant works in descent methods, where the KL inequality is used as a standard technique.
Indeed, we investigate the inexact versions of the PGA for solving the $\ell_p$ regularization problem \eqref{eq-lp}, in which the proximity operator of the $\ell_p$ regularizer \eqref{eq-PPA} is approximately solved at each iteration (with progressively better accuracy). Inspired by the ideas in the seminal work of Rockafellar \cite{roc76}, we consider two types of inexact PGAs: one measures the inexact term by the approximation of proximal regularized function value, and the other is measured by the distance of the iterate to the exact proximal operator (see Algorithms IPGA-I and IPGA-II). Under some suitable assumptions on the inexact terms, we establish the linear convergence of these two inexact PGAs to a local minimum of problem \eqref{eq-lp}; see Theorems \ref{thm-IPGA-I} and \ref{thm-IPGA-II}. It is worth noting that neither of these inexact PGAs satisfies the conditions of the inexact descent method mentioned earlier; see the explanation in Remark \ref{rem-IDM}(ii). In our analysis in this part, Theorem \ref{thm-LC} plays an important role in such a way that we are able to show that the components sequence on the support of the limiting point satisfies the conditions of Theorem \ref{thm-LC}.
We further propose two implementable 
inexact PGAs that satisfy the assumptions made in the convergence theorems and thus share the linear convergence property. 

As an interesting byproduct, the results obtained above are extended to the infinite-dimensional Hilbert spaces. Bredies et al. \cite{Bredies2015} investigated the PGA for solving the $\ell_p$ regularization problem in infinite-dimensional Hilbert spaces and proved its global convergence to a critical point under some technical assumptions and using dedicated tools from algebraic geometry; see the explanation before Theorem \ref{thm-PGA-inf}. Dropping these technical assumptions, we prove the global convergence of the PGA under the only assumption on stepsizes (as in \cite{Bredies2015}), which significantly improves \cite[Theorem 5.1]{Bredies2015}, and, under a simple additional assumption,
further establish the linear convergence of the descent method and PGA, as well as their inexact versions, for solving the $\ell_p$ regularization problem in infinite-dimensional Hilbert spaces.


The paper is organized as follows. In section 2, we present the notations and preliminary results to be used in the present paper. In section 3, we establish the equivalences among a local minimum, second-order optimality condition and second-order growth property of the $\ell_p$ regularization problem \eqref{eq-lp}, as well as some interesting corollaries. By virtue of the second-order optimality condition and second-order growth property, the linear convergence of an inexact descent method and inexact PGAs for solving problem \eqref{eq-lp} are established in sections 4 and 5, respectively. Finally, the convergence properties of relevant algorithms are extended to the infinite-dimensional Hilbert spaces in section 6.

\section{Notation and preliminary results}
We consider the $n$-dimensional Euclidean space $\R^n$ with inner product $\langle \cdot\, ,\cdot \rangle$ and Euclidean norm $\|\cdot\|$.
For $0 < p < 1$ and $x\in \R^n$, the $\ell_p$ ``norm"  on $\R^n$ is denoted by $\|\cdot\|_p$ and defined as follows:
\[
\|x\|_p:=\left(\sum_{i=1}^n |x_i|^p\right)^{\frac1p}\quad \mbox{for each } x\in \R^n;
\]
while  $\|x\|_0$   denotes the number of nonzero components of  $x$.
It is well-known (see, e.g., \cite[Eq. (7)]{HuJMLR17}) that
\begin{equation}\label{eq-norm}
\|x\|_p\ge \|x\|_q\quad \mbox{for each } x\in \R^n \mbox{ and } 0<p\le q.
\end{equation}
%
We write ${\rm supp}:\R^n\to \R$ and ${\rm sign}:\R \to \R$ to denote the support function and signum function, respectively.
For an integer $l\le n$, fixing $x\in \R^l$ and $\delta\in\R_+$,  we use $\mathbf{B}(x,\delta)$ to denote the open ball of radius $\delta$ centered at $x$ (in the Euclidean norm).
Moreover, we write
\[
\R^l_{\neq}:=\{x\in \R^l:x_i\neq 0 \mbox{ for each } i=1,\dots,l\}.
\]
Let $\R^{l\times l}$ denote  the space of all $l\times l$ matrices.  We endow $\R^{l\times l}$ with the partial orders $\succ$ and  $\succeq $, which are  defined for any $Y,\, Z\in \R^{l\times l}$ by
\[
Y \succ (\mbox{resp.}, \succeq) \, Z \quad \Longleftrightarrow  \quad Y-Z \mbox{ is positive definite (resp., positive semi-definite)}.
\]
Thus, for $Z\in\R^{l\times l}$,   $Z\succ 0$ (resp.,  $Z\succeq 0$, $Z\prec 0$) means  that $Z$ is positive definite (resp., positive semi-definite, negative definite). In particular, we use ${\rm diag}(x)$ to denote a square diagonal matrix with the components of vector $x$ on its main diagonal.

For simplicity, associated with problem \eqref{eq-lp}, we use $F:\R^n\to \R$ to denote the $\ell_p$ regularized function, and $H:\R^n\to \R$ and $\Phi:\R^n\to \R$ are the functions defined by
\begin{equation}\label{eq-Flp}
F(\cdot):=H(\cdot)+\Phi(\cdot), \quad H(\cdot):=\|A\cdot-b\|^2 \quad \mbox{and} \quad \Phi(\cdot):=\lambda\|\cdot\|_p^p.
\end{equation}
Letting $x^*\in \R^n\setminus \{0\}$, we write
\begin{equation}\label{eq-sI}
s:=\|x^*\|_0\quad \mbox{and} \quad I:={\rm supp}(x^*),
\end{equation}
We write $A_{i}$ to denote the $i$-th column of $A$, $A_I:=(A_i)_{i\in I}$ and $x_I:=(x_i)_{i\in I}$.
Let $f:\R^s\to \R$, $h:\R^s\to \R$ and $\varphi:\R^s\to \R$ be the functions defined by
\begin{equation}\label{eq-fun-g}
f(\cdot):=h(\cdot)+\varphi(\cdot), \quad h(\cdot):=\|A_I\cdot-b\|^2\quad \mbox{and} \quad \varphi(\cdot):=\lambda\|\cdot\|_p^p
\end{equation}
Obviously, $\varphi$ is smooth (of arbitrary order) on $ \R^s_{\neq}$, and so is $f$. The first- and second-order derivatives of $\varphi$ at each $y\in \R^s_{\neq}$ are respectively given by
\begin{equation}\label{eq-dg}
\nabla \varphi(y)=\lambda p\left(\left(|y_i|^{p-1}{\rm sign}(y_i)\right)_{i\in I}\right) \quad \mbox{and}\quad \nabla^2 \varphi(y)=\lambda p(p-1){\rm diag}\left(\left(|y_i|^{p-2}\right)_{i\in I}\right).
\end{equation}
Since $0<p<1$, it is clear that $\nabla^2\varphi(y)\prec 0$ for any $y\in \R_{\neq}^s$.
By \eqref{eq-Flp} and \eqref{eq-fun-g}, one sees that
\begin{equation}\label{eq-g-Fp}
\Phi(x)=\varphi(x_I) \mbox{ and } F(x)=f(x_I) \mbox{ for each $x$ satisfying ${\rm supp}(x)=I$}.
\end{equation}
The point $x^*$ is called a critical point of problem \eqref{eq-lp} if it satisfies that $\nabla f(x^*_I)=0$.
The following elementary equality is repeatedly used in our convergence analysis:
\begin{equation}\label{eq-Axb}
 \|Ay-b\|^2-\|Ax-b\|^2=\langle y-x, 2A^\top (Ax-b)\rangle +\|A(y-x)\|^2
\end{equation}
(by Taylor's formula applied to the function $\|A\cdot-b\|^2$).
We end this section by providing the following lemma, which is useful to establish the linear convergence of inexact decent methods.
\begin{lemma}\label{lem-LC-seq}
Let $\eta\in (0,1)$, and let $\{a_k\}$ and $\{\delta_k\}$ be two sequences of nonnegative scalars such that
\begin{equation}\label{eq-lem-seq}
a_{k+1}\le a_k \eta +\delta_k \mbox{ for each } k\in \IN \quad \mbox{and} \quad \limsup_{k\to \infty} \frac{\delta_{k+1}}{\delta_k}<1.
\end{equation}
Then there exist $\theta\in (0,1)$ and $K>0$ such that
\begin{equation}\label{eq-lem-seq-a}
a_k\le K \theta^k \quad \mbox{for each } k\in \IN.
\end{equation}
\end{lemma}
\begin{proof}
We first claim that there exist $\theta\in (0,1)$ and a sequence of nonnegative scalars $\{c_k\}$ such that
\begin{equation}\label{eq-lem-seq-ab}
a_{k+1}\le a_k \theta +c_k \theta^k \mbox{ for each } k\in \IN \quad \mbox{and} \quad \sum_{k=0}^\infty c_k<+\infty.
\end{equation}
Indeed, by the second inequality of \eqref{eq-lem-seq}, there exist $\tau\in (0,1)$ and $N\in \IN$ such that $\delta_{k+1}\le \tau^2 \delta_k$ for each $k\ge N$.
Letting $c_i:=\tau^{i-2N} \delta_N$ when $i\ge N$ and $c_i:=\frac{\delta_i}{\tau^i}$ otherwise, this shows that
\begin{equation}\label{eq-lem-seq-ad}
\delta_k\le c_k \tau^k  \quad \mbox{for each } k\in \IN.
\end{equation}
Consequently, we check that
$
\sum_{k=0}^\infty c_k=\sum_{k=0}^{N-1} c_k + \frac{\tau^{-N}}{1-\tau}\delta_N<+\infty.
$
Letting $\theta:=\max\{\eta,\tau\}$ and combining \eqref{eq-lem-seq} and \eqref{eq-lem-seq-ad}, we arrive at \eqref{eq-lem-seq-ab}, as desired.

Next, we show by mathematical induction that the following relation holds for each $k\in \IN$:
\begin{equation}\label{eq-lem-seq-1}
a_k\le \max\left\{1, \frac{a_1}{c_0+\theta}\right\}\prod_{i=0}^{k-1} (c_i+\theta).
\end{equation}
Clearly, \eqref{eq-lem-seq-1} holds for $k = 1$. Assuming that \eqref{eq-lem-seq-1} holds for each $k\le N$, we estimate $a_{N+1}$ in the following two cases.\\
Case 1. If $a_N<\theta^N$, it follows from the first inequality of \eqref{eq-lem-seq-ab} that
\[
a_{N+1}\le (\theta+c_N)\theta^N\le \prod_{i=0}^{N} (c_i+\theta)\le \max\left\{1, \frac{a_1}{c_0+\theta}\right\}\prod_{i=0}^{N} (c_i+\theta).
\]
Case 2. If $a_N\ge \theta^N$, one sees by \eqref{eq-lem-seq-ab} and \eqref{eq-lem-seq-1} (when $k=N$) that
\[
a_{N+1}\le (\theta+c_N)a_N\le \max\left\{1, \frac{a_1}{c_0+\theta}\right\}\prod_{i=0}^{N} (c_i+\theta).
\]
Hence, for both cases, \eqref{eq-lem-seq-1} holds for $k=N+1$, and so, it holds for each $k\in \IN$ by mathematical induction. Clearly, \eqref{eq-lem-seq-1} can be reformulated as
\begin{equation}\label{eq-lem-seq-2}
a_k\le\max\left\{1, \frac{a_1}{c_0+\theta}\right\}\theta^k\exp\left(\sum_{i=0}^{k-1} \ln\left(1+\frac{c_i}{\theta}\right) \right).
\end{equation}
Note that $\ln(1+t) \le t$ for any $t \ge 0$. It follows that
\[
\sum_{i=0}^{k-1} \ln(1+\frac{c_i}{\theta})\le \frac1{\theta}\sum_{i=0}^{k-1} c_i\le \frac1{\theta}\sum_{i=0}^{\infty} c_i<+\infty
\]
(by \eqref{eq-lem-seq-ab}). Letting $K:=\max\left\{1, \frac{a_1}{c_0+\theta}\right\}\exp\left(\frac1{\theta}\sum_{i=0}^{\infty} c_i\right)$, we conclude \eqref{eq-lem-seq-a} by \eqref{eq-lem-seq-2}, and the proof is complete.
\end{proof}

\section{Characterizations of local minima}
Optimality condition is a crucial tool for optimization problems, either providing the useful characterizations of (local) minima or designing effective optimization algorithms.
Some sufficient or necessary optimality conditions for the $\ell_p$ regularization problem \eqref{eq-lp} have been developed in the literature; see \cite{ChenXJ10,HuJMLR17,ZLu2013,Nikolova13} and references therein.
In particular, Chen et al. \cite{ChenXJ10} established the following first- and second-order necessary optimality conditions for a local minimum $x^*$ of problem \eqref{eq-lp}, i.e.,
\begin{equation}\label{eq-1st}
2A_I^{\top}(A_Ix^*_I-b)+\lambda p\left(\left(|x^*_i|^{p-1}{\rm sign}(x^*_i)\right)_{i\in I}\right)=0,
\end{equation}
and
\begin{equation}\label{eq-2nd-nec}
2A_I^{\top}A_I+\lambda p(p-1){\rm diag}\left(\left(|x^*_i|^{p-2}\right)_{i\in I}\right)\succeq 0,
\end{equation}
where $I={\rm supp}(x^*)$ is defined by \eqref{eq-sI}.
These necessary conditions were used to estimate the (lower/upper) bounds for the absolute values and  the number of nonzero components of local minima.
However, it seems that a complete optimality condition that is both necessary and sufficient for the local minima of the $\ell_p$ regularization problem has not been established yet in the literature. To remedy this gap, this section is devoted to providing some necessary and sufficient characterizations for the local minima of problem \eqref{eq-lp}.


To begin with, the following lemma (i.e., \cite[Lemma 10]{HuJMLR17}) illustrates that the $\ell_p$ regularized function satisfies a first-order growth property at $0$, 
which is useful for proving the equivalent characterizations of its local minima.
This property also indicates a significant advantage of the $\ell_p$ regularization over the $\ell_1$ regularization that the $\ell_p$ regularization has a strong sparsity promoting capability.

\begin{lemma}\label{lem-FOG}
Let $h:\R^n\rightarrow \R$ be a continuously differentiable function.
Then there exist $\epsilon>0$ and $\delta>0$ such that
\[
h(x)+\lambda \|x\|_p^p\ge h(0)+\epsilon\|x\|\quad \mbox{for any } x\in \mathbf{B}(0,\delta).
\]
\end{lemma}

The main result of this section is presented in the following theorem, in which we establish the equivalences among a local minimum, second-order optimality condition and second-order growth property of the $\ell_p$ regularization problem \eqref{eq-lp}. Note that the latter two conditions were provided in \cite{HuJMLR17} as necessary conditions for the group-wised $\ell_p$ regularization problem, while the second-order optimality condition is an improvement of the result in \cite{ChenXJ10} in that the matrix in the left-hand side of \eqref{eq-2nd-nec} is indeed positive definite.
Recall that $F:\R^n\to \R$ is the $\ell_p$ regularized function defined by \eqref{eq-Flp}  and $I={\rm supp}(x^*)$ is defined by \eqref{eq-sI}.


\begin{theorem}\label{thm-SOG}
Let $x^*\in \R^n\setminus \{0\}$.
Then the following assertions are equivalent:
\begin{enumerate}[{\rm (i)}]
\item $x^*$ is a local minimum of problem \eqref{eq-lp}.
\item \eqref{eq-1st} and the following condition hold:
\begin{equation}\label{eq-2nd}
2A_I^{\top}A_I+\lambda p(p-1){\rm diag}\left(\left(|x^*_i|^{p-2}\right)_{i\in I}\right)\succ 0.
\end{equation}
\item Problem \eqref{eq-lp} satisfies the second-order growth property at $x^*$, 
i.e., there exist $\epsilon>0$ and $\delta>0$ such that
\begin{equation}\label{eq-sgc}
F(x)\geq F(x^*)+\epsilon\|x-x^*\|^2\quad \mbox{for any } x\in \mathbf{B}(x^*,\delta).
\end{equation}
\end{enumerate}
\end{theorem}
\begin{proof}
Without loss of generality, we assume that $I=\{1,\dots,s\}$.

(i) $\Rightarrow$ (ii). Suppose that (i) holds. Then $x^*_I$ is a local minimum of $f$ (by \eqref{eq-g-Fp}), and \eqref{eq-1st} and \eqref{eq-2nd-nec} hold by \cite[pp. 76]{ChenXJ10} (they can also be checked directly by the optimality condition for  smooth optimization in \cite[Proposition 1.1.1]{Bertsekas99}):
$\nabla f(x_I^*)=0$ and $\nabla^2 f(x_I^*)\succeq 0$.
Thus, it remains to prove \eqref{eq-2nd}, i.e., $\nabla^2 f(x_I^*)\succ 0$.
  To do this, suppose on the contrary that \eqref{eq-2nd} does not hold.
Then, by  \eqref{eq-2nd-nec}, there exists $w\neq 0$ such that $\langle w, \nabla^2 f(x^*_I) w\rangle=0$.
Let $\psi:\R\to \R$ be defined by 
\[\psi(t):=f(x^*_I+tw)\quad \mbox{for each } t\in \R.\]
Then one sees that $\psi'(0)=\langle w, \nabla f(x^*_I)\rangle=0$ and $\psi''(0)=\langle w, \nabla^2 f(x^*_I)w\rangle=0$, and
 $0$ is a local minimum of $\psi$ (as $x^*_I$ is a local minimum of $f$). Therefore, $\psi^{(3)}(0)=0$ and $\psi^{(4)}(0)\ge 0$.
However, by the elementary   calculus, one can check  that
\begin{equation*}\label{eq-4th-nes}
\psi^{(4)}(0)=\lambda p (p-1)(p-2)(p-3) \sum_{i\in I} \left(w_i^4|x_i^*|^{p-4}\right)<0,
\end{equation*}
which yields a contradiction. Hence, assertion (ii) holds.

(ii) $\Rightarrow$ (iii). Suppose that assertion (ii) of this theorem holds. Then
    \begin{equation}\label{1-2suf-c}
 \nabla f(x_I^*)=0\quad\mbox{and}\quad \nabla^2 f(x_I^*)\succ 0.
\end{equation}
By Taylor's formula, we have that
$$f(y)= f(x_I^*)+ \nabla f(x_I^*)(y-x_I^*)+ \frac12 \langle y-x_I^*, \nabla^2 f(x_I^*)(y-x_I^*)\rangle +o(\|y-x_I^*\|^2)\quad \mbox{for each } y\in \R^s.
$$
This, together with \eqref{1-2suf-c},  implies   that there exist $\epsilon_1>0$ and $\delta_1>0$ such that
\begin{equation}\label{eq-sog-0a}
f(y)\geq f(x_I^*)+2\epsilon_1\|y-x_I^*\|^2\quad \mbox{for any } y\in \mathbf{B}(x_I^*,\delta_1).
\end{equation}
Let $\tau>0$ be such that $\sqrt{\epsilon_1\tau}\ge \|A_I\|\|A_{I^c}\|$,
and define $g:\R^{n-s}\to \R$   by
\begin{equation}\label{eq-fun-h}
g(z):=\|A_{I^c}z\|^2+2\langle A_Ix^*_I-b,  A_{I^c}z \rangle-2\tau\|z\|^2\quad \mbox{for each } z\in \R^{n-s}.
\end{equation}
Clearly, $g$ is continuously differentiable on $\R^{n-s}$ with  $g(0)=0$. Then, by Lemma \ref{lem-FOG}, there exist $\epsilon_2>0$ and $\delta_2>0$ such that
\begin{equation}\label{eq-sog-0e}
g(z)+\lambda\|z\|_p^p\ge g(0)+\epsilon_2\|z\|=\epsilon_2\|z\|\geq 0\quad \mbox{for any } z\in \mathbf{B}(0,\delta_2).
\end{equation}
Fix $x:=
\left(\begin{matrix}
   x_I\\
   x_{I^c}\\
\end{matrix}\right)$
with $x_I\in \mathbf{B}(x^*_I,\delta_1)$ and $x_{I^c}\in \mathbf{B}(0,\delta_2)$. Then it follows from the definitions of the functions $F$, $f$ and $g$
(see \eqref{eq-Flp}, \eqref{eq-fun-g} and \eqref{eq-fun-h})  that
\begin{equation*}
\begin{array}{ll}
F(x)&=\|A_Ix_I+A_{I^c}x_{I^c}-b\|^2+\lambda\|x_I\|_p^p+\lambda\|x_{I^c}\|_p^p\\
&=\|A_Ix_I-b\|^2+\|A_{I^c}x_{I^c}\|^2+2\langle A_Ix_I-b,  A_{I^c}x_{I^c}\rangle+\lambda\|x_I\|_p^p+\lambda\|x_{I^c}\|_p^p\\
&=f(x_I)+g(x_{I^c})+2\tau \|x_{I^c}\|^2+\lambda\|x_{I^c}\|_p^p+2\langle A_I(x_I-x^*_I),  A_{I^c}x_{I^c}\rangle.
\end{array}
\end{equation*}
Applying  \eqref{eq-sog-0a}   (to  $x_I$ in place of $y$) and \eqref{eq-sog-0e} (to $x_{I^c}$ in place of $z$), we have that
\begin{equation*}\label{eq-sog-0c}
\begin{array}{ll}
F(x)\ge f(x^*_I)+2\epsilon_1\|x_I-x^*_I\|^2+2\tau\|x_{I^c}\|^2+2\langle A_I(x_I-x^*_I),  A_{I^c}x_{I^c}\rangle.
\end{array}
\end{equation*}
By the definition of $\tau$, we have that
\[
 2|\langle A_I(x_I-x^*_I),  A_{I^c}x_{I^c}\rangle|\le 
2\sqrt{\epsilon_1\tau}\|x_I-x^*_I\|\|x_{I^c}\| \le \epsilon_1\|x_I-x^*_I\|^2+\tau\|x_{I^c}\|^2,
\]
and then, it follows that
\[
F(x) \ge f(x^*_I)+\epsilon_1\|x_I-x^*_I\|^2+\tau\|x_{I^c}\|^2
 \ge f(x^*_I)+\min\{\epsilon_1,\tau\}\|x-x^*\|^2
\]
(noting that $x_{I^c}=x_{I^c}-x^*_{I^c}$).
Hence $F(x)\ge  F(x^*)+\min\{\epsilon_1,\tau\}\|x-x^*\|^2,$
as $f(x^*_I)=F(x^*)$ by \eqref{eq-g-Fp}.  This means that \eqref{eq-sgc} holds with
  $\epsilon:=\min\{\epsilon_1,\tau\}$ and $\delta:=\min\{\delta_1,\delta_2\}$, and so (iii) is verified.

(iii) $\Rightarrow$ (i).  It is trivial.  The proof is complete.
\end{proof}

\begin{remark}
As shown in Lemma \ref{lem-FOG},  for the case when $x^*=0$, the equivalence between assertions (i) and (iii) in Theorem \ref{thm-SOG} is true, while assertion (ii) is not well defined (as $I=\emptyset$).
\end{remark}


The structure of local minima is a useful property for the numerical study of the $\ell_p$ regularization problem; see, e.g., \cite{ChenXJ10,XuZB12}.
As a byproduct of Theorem \ref{thm-SOG}, we will prove that the number of local minima of problem \eqref{eq-lp} is finite, which was claimed in  \cite[Corollary 2.2]{ChenXJ10} but with an incomplete proof (because their proof is based on the fact that $f$ has at most one local minimum whenever $A_I^\top A_I$ is of full rank, which is unclear).

\begin{corollary}\label{coro-FM}
 The $\ell_p$ regularization problem \eqref{eq-lp} has only a finite number of local minima.
\end{corollary}
\begin{proof}
Let  $I\subseteq \{1,\dots,n\}$. We use  ${\rm LM}(F,\R^n;I)$ to denote the set of local minima $x^*$ of problem \eqref{eq-lp} with ${\rm supp}(x^*)=I$, and set
\begin{equation}\label{eq-FM-theta}
\Theta(I):=\left\{x_I:x\in {\rm LM}(F,\R^n;I)\right\}.
\end{equation}
Then the set of local minima of problem \eqref{eq-lp} can be expressed as the union of ${\rm LM}(F,\R^n;I)$ over all subsets  $I\subseteq \{1,\dots,n\}$.
Clearly,   ${\rm LM}(F,\R^n;I)$ and $\Theta(I)$ have the same cardinality. Thus, to complete the proof,
 it suffices to show that $\Theta(I)$ is finite.
To do this,  we may assume that, without loss of generality,  $I=\{1,\dots,s\}$, and write
\begin{equation}\label{eq-FM-O}
O:=\{y\in \R^s_{\neq}: \nabla^2 f(y)\succ 0\},
\end{equation}
where  $f:\R^s\to \R$ is defined by \eqref{eq-fun-g}.
Clearly, $O$ is open in $\R^s$, and  $\Theta(I)\subseteq O$ by Theorem \ref{thm-SOG}. Thus, it follows   from \eqref{eq-FM-theta} that
\begin{equation}\label{eq-FM-1}
\Theta(I)\subseteq {\rm LM}(f,\R^s)\cap O
\end{equation}
(we indeed can show an equality), where, for an open subset $U$ of $\R^s$, ${\rm LM}(f,U)$   stands for the set of local minima of $f$ over $U$.
For simplicity, we set
 \begin{equation*}\label{eq-OJ}
  \R^s_J:=\{y\in \R^s: y_j>0 \mbox{ for }j\in J,\, y_j<0 \mbox{ for } j\in I\setminus J\}
  \end{equation*}
 and  $O_J:=O\cap\R^s_J$  for  any $J\subseteq I$. Then each $O_J$ is open in  $\R^s$ (as so are $O$ and $\R^s_J$). This particularly implies that
 \begin{equation}\label{xjia}
       {\rm LM}(f,\R^s)\cap O_J =    {\rm LM}(f,O_J) \quad\mbox{for each }J\subseteq I.
       \end{equation}
 Moreover, it is clear that $O=\cup_{J\subseteq I} O_J$. Hence
     \begin{equation}\label{eq-OJ-1}
     \Theta(I)\subseteq {\rm LM}(f,\R^s) \cap O =\cup_{J\subseteq I}\, \left( {\rm LM}(f,\R^s)\cap O_J\right)
     =\cup_{J\subseteq I}\,   {\rm LM}(f, O_J)
     \end{equation}
(thanks to   \eqref{eq-FM-1} and \eqref{xjia}).
%
     Below we show  that
\begin{equation}\label{eq-OJ-2}
O_J \mbox{ is  convex for each } J\subseteq I.
\end{equation}
Granting this, one concludes that   each $ {\rm LM}(f, O_J)$ is at most a singleton, because $\nabla^2 f\succ 0$ on $ O_J$ by \eqref{eq-FM-O} and then
$f$ is strictly convex on $O_J$ by the higher-dimensional derivative tests for convexity (see, e.g., \cite[Theorem 2.14]{Roc98});
hence $\Theta(I)$ is finite by \eqref{eq-OJ-1}, completing the  proof.

To show \eqref{eq-OJ-2}, fix  $J\subseteq I$, and let $y, z\in O_J$. Then, by definition, one has that
 \begin{equation}\label{eq-FM-2}
 \nabla^2 f(y)\succ 0 \quad \mbox{and} \quad \nabla^2 f(z)\succ 0.
 \end{equation}
%
%
%
By   elementary calculus, the map $t\mapsto t^{p-2}$ is convex on $(0, +\infty)$,
and so
\begin{equation*}\label{dddr}
 \frac{|y_i|^{p-2}+|z_i|^{p-2}}2\ge \left(\frac  {|y_i|+|z_i|}2\right)^{p-2}\quad\mbox{for each }i\in I.
\end{equation*}
 Consequently, we have
\[
{\rm diag}\left(\left(\frac{|y_i|^{p-2}+|z_i|^{p-2}}2\right)_{i\in I}\right)\succeq
{\rm diag}\left(\left(\left(\frac  {|y_i|+|z_i|}2\right)^{p-2}\right)_{i\in I}\right).
\]
This, together with \eqref{eq-dg} and \eqref{eq-FM-2}, implies that
\[
\nabla^2 f\left(\frac {y+z}2\right)\succeq \frac{\nabla^2 f(y) +\nabla^2 f(z)}2 \succ 0.
\]
Since $\frac {y+z}2\in  \R^s_J \subseteq  \R^s_{\not=}$, it follows that  $\frac {y+z}2\in O\cap \R^s_J=O_J$ and   \eqref{eq-OJ-2} is proved.
\end{proof}

Another byproduct of Theorem \ref{thm-SOG} is the following corollary, in which we show the isolation of a local minimum of problem \eqref{eq-lp} in the sense of critical points. This property is useful for establishing the global convergence of the inexact descent method and inexact PGA.

\begin{corollary}\label{coro-iso}
Let $x^*$ be a local minimum of the $\ell_p$ regularization problem \eqref{eq-lp}. Then $x^*$ is an isolated critical point of problem \eqref{eq-lp}.
\end{corollary}
\begin{proof}
Recall that $I={\rm supp}(x^*)$ and $f$ are defined by \eqref{eq-sI} and \eqref{eq-fun-g}, respectively.
Since $x^*$ is a local minimum of problem \eqref{eq-lp}, it follows from \eqref{eq-g-Fp} that $x^*_I$ is a local minimum of $f$ and
from Theorem \ref{thm-SOG} (cf. \eqref{eq-2nd}) that $\nabla^2 f(x^*_I)\succ 0$.
By the fact that $x^*_I\in \R_{\neq}^s$ and by the smoothness of $f$ at $x_I^*$, we can find a constant $\tau$ with
\begin{equation}\label{eq-iso-a2}
0< \tau<\left(\frac{4}{\lambda p}\|A^\top(Ax^*-b)\|_\infty\right)^{\frac1{p-1}}
\end{equation}
such that
\begin{equation}\label{eq-iso-a}
\mathbf{B}(x^*_I,\tau)\subseteq \R_{\neq}^s\cap \{y\in \R^s:\nabla^2 f(y)\succ 0\}.
\end{equation}
We aim to show that $\mathbf{B}(x^*,\tau)$ includes only one critical point of problem \eqref{eq-lp}, that is $x^*$. To do this, let $x\in \mathbf{B}(x^*,\tau)$ be a critical point of problem \eqref{eq-lp}. We first claim that ${\rm supp}(x)=I$.
It is clear by \eqref{eq-iso-a} that
\begin{equation}\label{eq-supp-2}
x_i \neq 0 \quad \mbox{when } i\in I, \mbox{ and } |x_i|<\tau \mbox{ otherwise.}
\end{equation}
If $i\in {\rm supp}(x)$,
by the definition of critical point, it follows that $2A_i^{\top}(Ax-b)+\lambda p|x_i|^{p-1}{\rm sign}(x_i)=0$;
consequently, by the fact that $x$ is closed to $x^*$, we obtain that
\[
|x_i|=\left(\frac{2|A_i^{\top}(A x-b)|}{\lambda p}\right)^{\frac1{p-1}}>\left(\frac{4|A_i^{\top}(A x^*-b)|}{\lambda p}\right)^{\frac1{p-1}}> \left(\frac{4\|A^{\top}(A x^*-b)\|_\infty}{\lambda p}\right)^{\frac1{p-1}}
> \tau
\]
(due to \eqref{eq-iso-a2}). This, together with \eqref{eq-supp-2}, shows that ${\rm supp}(x)=I$, as desired.

Finally, we show that $x=x^*$. By \eqref{eq-iso-a}, one has that $f$ is strongly convex on $\mathbf{B}(x^*_I,\tau)$. Since $x$ is a critical point of problem \eqref{eq-lp}, one has by the definition of critical point that $\nabla f(x_I)=0$, and so $x_I$ is a minimum of $f$ on $\mathbf{B}(x^*_I,\tau)$. By the strongly convexity of $f$ on $\mathbf{B}(x^*_I,\tau)$, we obtain $x_I=x^*_I$, and hence that $x=x^*$ (since ${\rm supp}(x)=I$). The proof is complete.
\end{proof}

\section{Linear convergence of inexact descent method}
 This section aims to establish the linear convergence of an inexact version of descent methods in a general framework. In our analysis, we will employ both second-order optimality condition and second-order growth property, established in Theorem \ref{thm-SOG}.

Let $\alpha$ and $\beta$ be fixed positive constants and $\{\epsilon_k\}\subseteq \R_+$ be a sequence of nonnegative scalars, and recall that $F:\R^n\to \R$ is the $\ell_p$ regularized function defined by \eqref{eq-Flp}. We consider a sequence $\{x^k\}$ that satisfies the following relaxed conditions of (H1) and (H2).
\begin{enumerate}
  \item[(H1$^{\circ}$)] For each $k\in \IN$,
  \begin{equation}\label{asp-H1a}
  F(x^{k+1})-F(x^k)\le -\alpha\|x^{k+1}-x^k\|^2+\epsilon_k^2;
  \end{equation}
  \item[{(H2$^{\circ}$)}] For each $k\in \IN$, there exists $w^{k+1} \in \partial F(x^{k+1})$ such that
  \begin{equation*}\label{asp-H2a}
  \|w^{k+1}\| \le \beta\|x^{k+1}-x^k\|+\epsilon_k.
  \end{equation*}
\end{enumerate}
Frankel et al. \cite{Frankel-JOTA} proposed an inexact version of descent methods, in which only (H2) is relaxed to the inexact form (H2$^{\circ}$) while the exact form (H1) is maintained; consequently, the sequence $\{x^k\}$ satisfies a descent property. However, in our framework, note by \eqref{asp-H1a} that the sequence $\{x^k\}$ does not satisfy a descent property. This is an essential difference from \cite{Frankel-JOTA} and extensive studies in descent methods.

We begin with the following useful properties of the inexact descent method;
in particular, a consistent property that $x^k$ has the same support as $x^*$ when $k$ is large (assertion (ii)) is useful for providing a uniform decomposition of $\{x^k\}$ in convergence analysis.

\begin{proposition}\label{lem-supp}
{\rm (i)} Let $\{x^k\}$ be a sequence satisfying ${\rm (H1^{\circ})}$ with
\begin{equation}\label{eq-epi}
\sum_{k=0}^\infty \epsilon_k^2<+\infty.
\end{equation}
Then $\sum_{k=0}^\infty \|x^{k+1}-x^k\|^2< +\infty$.

{\rm (ii)} Let $\{x^k\}$ be a sequence satisfying ${\rm (H2^{\circ})}$ with $\lim_{k\to \infty} \epsilon_k=0$. Suppose that $\{x^k\}$ converges to $x^*$.
Then there exists $N\in \IN$ such that
\begin{equation}\label{eq-lem-supp}
{\rm supp}(x^k)={\rm supp}(x^*) \quad \mbox{for each } k\ge N.
\end{equation}
\end{proposition}
\begin{proof}
Assertion (i) of this theorem is trivial by the assumption and the fact that $F\ge 0$.
Below, we prove assertion (ii).
Write
\begin{equation}\label{eq-supp-4a}
\gamma:=\left(\frac{\lambda p}{\beta+1+4\|A^\top(Ax^*-b)\|_\infty}\right)^{\frac1{1-p}}.
\end{equation}
By the assumption that $\{x^k\}$ converges to $x^*$, there exists $N\in \mathbb{N}$ such that for each $k\ge N$
\begin{equation}\label{eq-supp-4b}
x^k_i \neq 0 \quad \mbox{when } i\in {\rm supp}(x^*), \mbox{ and } |x^k_i|<\gamma \mbox{ otherwise.}
\end{equation}
Fix $k \ge N$ and $i\in {\rm supp}(x^k)$.
By the assumption (H2$^\circ$), there exists $w^k\in \partial F(x^k)$ such that
\begin{equation}\label{eq-supp-4c}
\|w^k\|\le \beta\|x^k-x^{k-1}\|+\epsilon_k<\beta+1
\end{equation}
(by the assumptions that $\lim_{k\to \infty} \epsilon_k=0$ and $\lim_{k\to \infty} x^k=x^*$).
Noting that $i\in {\rm supp}(x^k)$, we obtain by \eqref{eq-dg} that
\begin{equation*}\label{eq-supp-4d}
|w_i^k| =|2A_i^{\top}(Ax^k-b)+\lambda p|x_i^k|^{p-1}{\rm sign}(x_i^k)| \ge \lambda p|x_i^k|^{p-1}-4\|A^\top(Ax^*-b)\|_\infty.
\end{equation*}
This, together with \eqref{eq-supp-4c} and \eqref{eq-supp-4a}, shows that $|x_i^k| > \gamma$ when $i\in {\rm supp}(x^k)$.
This, together with \eqref{eq-supp-4b}, shows that ${\rm supp}(x^k)={\rm supp}(x^*)$ for each $k\ge N$. The proof is complete.
\end{proof}

The main theorem of this section is as follows.
The convergence theorem is not only of independent interest in establishing the linear convergence of inexact descent method, but also provides a useful approach for the linear convergence study of the inexact PGA in the next section. Recall that functions $F$ and $f$ are defined by \eqref{eq-Flp} and \eqref{eq-fun-g}, respectively.

\begin{theorem}\label{thm-LC}
Let $\{x^k\}$ be a sequence satisfying ${\rm (H1^{\circ})}$ and $\{\epsilon^k\}$ satisfy \eqref{eq-epi}.
Suppose one of limiting points of $\{x^k\}$, denoted by $x^*$, is a local minimum of problem \eqref{eq-lp}. Then the following assertions are true.
\begin{enumerate}[\rm (i)]
  \item $\{x^k\}$ converges to $x^*$.
  \item Suppose further that $\{x^k\}$ satisfies ${\rm (H2^{\circ})}$ and
  \begin{equation}\label{eq-epsilon}
    \limsup_{k\to \infty} \frac{\epsilon_{k+1}}{\epsilon_k}<1.
    \end{equation}
    Then $\{x^k\}$ converges linearly to $x^*$, that is, there exist $C>0$ and $\eta\in (0,1)$ such that
\begin{equation}\label{eq-thm-LC}
F(x^{k})-F(x^*)\le C\eta^k\quad {\rm and} \quad \|x^k-x^*\|\le C\eta^k\quad \mbox{for each } k\in \IN.
\end{equation}
\end{enumerate}
\end{theorem}
\begin{proof}
(i) It follows from Proposition \ref{lem-supp}(i) that $\lim_{k\to \infty} \|x^{k+1}-x^k\|=0$.
By the assumption that $x^*$ is a local minimum of problem \eqref{eq-lp}, it follows from Lemma \ref{coro-iso} that $x^*$ is an isolated critical point of problem \eqref{eq-lp}. Then, we can prove that $\{x^k\}$ converges to $x^*$ (the proof is standard; see, e.g., the proof of \cite[Proposition 2.3]{Bredies2015}).

(ii) If $x^*=0$, it follows from Proposition \ref{lem-supp}(ii) that there exists $N\in \IN$ such that $x^k=0$ for each $k\ge N$, and so the conclusion holds. Then it remains to prove \eqref{eq-thm-LC} for the case when $x^*\neq 0$.

Suppose that $x^*\neq 0$. Recall that $I={\rm supp}(x^*)$ is defined by \eqref{eq-sI}. By the assumption that $x^*$ is a local minimum of problem \eqref{eq-lp}, assertions (ii) and (iii) of Theorem \ref{thm-SOG} are satisfied; hence, it follows from \eqref{eq-2nd} and \eqref{eq-dg} that
$2A_I^{\top}A_I+\nabla^2 \varphi(x^*_I)=\nabla^2 f(x^*_I)\succ 0$.
This, together with $x^*_I\in \R_{\neq}^s$ (cf. \eqref{eq-sI}) and the smoothness of $\varphi$ at $x_I^*$, implies that there exist $\epsilon>0$, $\delta>0$ and $L_{\varphi}>0$ such that \eqref{eq-sgc} holds and
\begin{equation}\label{eq-LC-2a}
\mathbf{B}(x^*_I,\delta)\subseteq \R_{\neq}^s\cap \{y\in \R^s:\nabla^2 \varphi(y)\succ -2A_I^\top A_I\},
\end{equation}
\begin{equation*}\label{eq-thm-n1}
\|\nabla \varphi(y)-\nabla \varphi(z)\|\le L_{\varphi}\|y-z\|\quad \mbox{for any } y,z\in \mathbf{B}(x_I^*,\delta).
\end{equation*}
By assertion (i) of this theorem that $\{x^k\}$ converges to $x^*$, there exists $N\in \IN$ such that \eqref{eq-lem-supp} holds (by Proposition \ref{lem-supp}(ii)) and $x_I^k\in \mathbf{B}(x^*_I,\delta)$ for each $k\ge N$.
In particular, the following relations hold for each $k\ge N$:
\begin{equation}\label{eq-thm-1a}
F(x^{k+1})\ge F(x^*)+\epsilon\|x^{k+1}-x^*\|^2,
\end{equation}
and
\begin{equation}\label{eq-thm-1}
\|\nabla \varphi(x_I^k)-\nabla \varphi(x_I^{k+1})\|\le L_{\varphi}\|x_I^k-x_I^{k+1}\|.
\end{equation}
Noting by  \eqref{eq-dg} and \eqref{eq-LC-2a} that
\[
\nabla^2 \varphi(w)\prec 0\quad \mbox{and}\quad \nabla^2 f(w)\succ 0\quad \mbox{ for any } w\in \mathbf{B}(x^*_I,\delta),
\]
it follows that $\varphi$ is concave and $f$ is convex on $\mathbf{B}(x^*_I,\delta)$.
Fix $k\ge N$. Then one has that
\begin{equation}\label{concave}
  \langle \nabla \varphi(x_I^k), x_I^k-x_I^{k+1}\rangle\le\varphi(x_I^k)-\varphi(x_I^{k+1})
\end{equation}
and
\begin{equation}\label{convex}
 f(x_I^k)-f(x^*_I)\le \langle \nabla f(x_I^k), x_I^k-x^*_I \rangle
\end{equation}
(as $x_I^k, x_I^{k+1}\in \mathbf{B}(x^*_I,\delta)$). 
To proceed, we define
\begin{equation}\label{eq-LC-b0a}
r_k:=F(x^{k})-F(x^*)\quad \mbox{for each } k\in \IN,
\end{equation}
and then it follows from \eqref{eq-lem-supp} and \eqref{eq-g-Fp} that
\begin{equation}\label{eq-LC-b0}
r_k=f(x_I^k)-f(x^*_I).
\end{equation}
Hence,  using  \eqref{convex}, we obtain  that
\begin{equation}\label{eq-LC-b1}
r_k\le \langle \nabla f(x_I^k), x_I^k-x^*_I \rangle= \langle \nabla f(x_I^k), x_I^k-x_I^{k+1} \rangle + \langle \nabla f(x_I^k), x_I^{k+1}-x^*_I \rangle.
\end{equation}
By \eqref{eq-fun-g} and \eqref{concave}, it follows that
\begin{equation*}\label{eq-LC-b6}
\begin{array}{lll}
\langle \nabla f(x_I^k), x_I^k-x_I^{k+1} \rangle&= \langle \nabla h(x_I^k), x_I^k-x_I^{k+1} \rangle+\langle\nabla \varphi(x_I^k), x_I^k-x_I^{k+1} \rangle\\
&\le \langle \nabla h(x_I^k), x_I^k-x_I^{k+1} \rangle+\varphi(x_I^k)-\varphi(x_I^{k+1}).
\end{array}
\end{equation*}
Recall from \eqref{eq-fun-g} that $\nabla h(x_I^k)=2A_I^\top(A_Ix_I^k-b)$. Then, by \eqref{eq-Axb} (with $A_I$, $x_I^{k+1}$, $x_I^{k+1}$ in place of $A$, $y$, $x$), we have that
\begin{equation}\label{eq-LC-b7}
\begin{array}{lll}
\langle \nabla f(x_I^k), x_I^k-x_I^{k+1} \rangle
&\le f(x_I^k)-f(x_I^{k+1})+\|A_I(x_I^{k+1}-x_I^k)\|^2\\
&\le r_k-r_{k+1}+\|A\|^2\|x^{k+1}-x^k\|^2
\end{array}
\end{equation}
(due to \eqref{eq-LC-b0}). On the other hand, one has that
\begin{equation}\label{eq-LC-b1c00}
\begin{array}{lll}
\langle \nabla f(x_I^k), x_I^{k+1}-x^*_I \rangle=\langle \nabla f(x_I^{k+1}), x_I^{k+1}-x^*_I \rangle+\langle \nabla f(x_I^k)-\nabla f(x_I^{k+1}), x_I^{k+1}-x^*_I \rangle.
\end{array}
\end{equation}
By the assumption (H2$^\circ$), we obtain that
\[
\begin{array}{lll}
\langle \nabla f(x_I^{k+1}), x_I^{k+1}-x^*_I \rangle&\le \|\nabla f(x_I^{k+1})\| \| x_I^{k+1}-x^*_I\|\\
&\le   \|w^{k+1}\|\|x_I^{k+1}-x^*_I\|\\
&\le  \beta \|x^{k+1}-x^k\|\|x^{k+1}-x^*\|+\epsilon_k\|x^{k+1}-x^*\|;
\end{array}
\]
while by \eqref{eq-fun-g} and \eqref{eq-thm-1}, we conclude  that
\begin{equation*}\label{eq-thm-n4+}
\begin{array}{lll}
&\langle \nabla f(x_I^k)-\nabla f(x_I^{k+1}), x_I^{k+1}-x^*_I \rangle\\
&=\langle \nabla h(x_I^k)-\nabla h(x_I^{k+1})+ \nabla \varphi(x_I^k)-\nabla \varphi(x_I^{k+1}), x_I^{k+1}-x^*_I \rangle\\
&\le (2\|A\|^2+L_{\varphi})\|x_I^{k+1}-x_I^k\|\|x_I^{k+1}-x^*_I\|\\
&\le (2\|A\|^2+L_{\varphi})\|x^{k+1}-x^k\|\|x^{k+1}-x^*\|.
\end{array}
\end{equation*}
Combining the above two inequalities, it follows from \eqref{eq-LC-b1c00} that
\[
\langle \nabla f(x_I^k), x_I^{k+1}-x^*_I \rangle\le \left(\beta+2\|A\|^2+L_{\varphi}\right) \|x^{k+1}-x^k\|\|x^{k+1}-x^*\|+\epsilon_k\|x^{k+1}-x^*\|.
\]
Let
\begin{equation}\label{eq-LC-c1}
\sigma:=\beta+2\|A\|^2+L_{\varphi}\quad \mbox{and}\quad \tau\in (0,\epsilon).
\end{equation}
Then one has that
\begin{eqnarray}
\langle \nabla f(x_I^k), x_I^{k+1}-x^*_I \rangle&\le& \frac{\sigma^2}{2\tau}\|x^{k+1}-x^k\|^2+\frac{\tau}{2}\|x^{k+1}-x^*\|^2
+\frac{1}{2\tau}\epsilon_k^2+\frac{\tau}{2}\|x^{k+1}-x^*\|^2\nonumber \\
&=& \frac{\sigma^2}{2\tau}\|x^{k+1}-x^k\|^2+\tau\|x^{k+1}-x^*\|^2+\frac{1}{2\tau}\epsilon_k^2. \nonumber
\end{eqnarray}
This, together with \eqref{eq-LC-b1} and \eqref{eq-LC-b7}, shows  that
\begin{equation}\label{eq-LC-b2}
r_k\le r_k-r_{k+1}+\left(\|A\|^2+\frac{\sigma^2}{2\tau}\right)\|x^{k+1}-x^k\|^2+\tau\|x^{k+1}-x^*\|^2+\frac{1}{2\tau}\epsilon_k^2.
\end{equation}
Recalling \eqref{eq-LC-b0a}, we obtain by the assumption (H1$^\circ$) that
\[
\|x^{k+1}-x^k\|^2\le \frac1{\alpha}\left(F(x^k)-F(x^{k+1})\right)+\frac{1}{\alpha}\epsilon_k^2=\frac1{\alpha}(r_k-r_{k+1})+\frac{1}{\alpha}\epsilon_k^2,
\]
and by \eqref{eq-thm-1a} that
\[
\|x^{k+1}-x^*\|^2\le \frac1{\epsilon}\left(F(x^{k+1})-F(x^*)\right)=\frac1{\epsilon}r_{k+1}.
\]
Hence, \eqref{eq-LC-b2} reduces to
\[
r_k\le r_k-r_{k+1}+\frac{2\tau\|A\|^2+\sigma^2}{2\tau\alpha}(r_k-r_{k+1})+\frac{2\tau\|A\|^2
+\sigma^2}{2\tau\alpha}\epsilon_k^2+\frac{\tau}{\epsilon}r_{k+1}+\frac{1}{2\tau}\epsilon_k^2,
\]
that is,
\begin{equation}\label{eq-LC-b4}
r_{k+1}\le \left(1-\frac{1-\frac{\tau}{\epsilon}}{1+\frac{2\tau\|A\|^2+\sigma^2}{2\tau\alpha}-\frac{\tau}{\epsilon}}\right)r_k
+\left(\frac{2\tau\|A\|^2+\sigma^2+\alpha}{2\tau\alpha+2\tau\|A\|^2+\sigma^2-2\tau^2\alpha\frac{1}{\epsilon}}\right)\epsilon_k^2.
\end{equation}
Let
$$\bar{\eta}:=1-\frac{1-\frac{\tau}{\epsilon}}{1+\frac{2\tau\|A\|^2+\sigma^2}{2\tau\alpha}-\frac{\tau}{\epsilon}}
\quad \mbox{and} \quad
\bar{c}:=\frac{2\tau\|A\|^2+\sigma^2+\alpha}{2\tau\alpha+2\tau\|A\|^2+\sigma^2-2\tau^2\alpha\frac{1}{\epsilon}}.$$
Then \eqref{eq-LC-b4} reduces to 
\[
r_{k+1}\le \bar{\eta} r_k + \bar{c} \epsilon^2_k\quad \mbox{for each } k\ge N.
\]
One can check that $0< \bar{\eta}<1$ and $\bar{c}>0$ by \eqref{eq-LC-c1}, and note \eqref{eq-epsilon}. Applying Lemma \ref{lem-LC-seq} (with $r_k$, $\bar{\eta}$ and $\bar{c} \epsilon^2_k$ in place of $a_k$, $\eta$ and $\delta_k$), there exist $\theta\in (0,1)$ and $K>0$ such that
\begin{equation*}\label{eq-LC-4a}
F(x^{k})-F(x^*)=r_k\le K \theta^k \quad \mbox{for each } k\ge N
\end{equation*}
(by \eqref{eq-LC-b0a}).
Furthermore, using \eqref{eq-thm-1a}, we have that
\[
\|x^k-x^*\|\le \left(\frac{F(x^{k})-F(x^*)}{\epsilon}\right)^{\frac12}\le \left(\frac{K}{\epsilon}\right)^{\frac12}  \left(\sqrt{\theta}\right)^k\quad \mbox{for each } k\ge N.
\]
This shows that  \eqref{eq-thm-LC} holds  with $C:=\max\left\{K,\left(\frac{K}{\epsilon}\right)^{\frac12}\right\} $ and $\eta:=\sqrt{\theta}$. The proof is complete.
\end{proof}

\begin{remark}
{\rm
It is worth noting in \eqref{eq-thm-LC} that the linear convergence of $\{F(x^{k})\}$ to $F(x^*)$ is a direct consequence of that of $\{x^{k}\}$ to $x^*$. Indeed, recalling from \cite[Lemma 2]{HuJMLR17} that $\|x\|_p^p-\|y\|_p^p\le \|x-y\|_p^p$ for any $x,y \in \R^n$, we obtain by \eqref{eq-Flp} that
\[
F(x^{k})-F(x^*)\le \|A\|^2\|x^{k}-x^*\|^2+\lambda\|x^{k}-x^*\|_p^p.
\]
}
\end{remark}

As an application of Theorem \ref{thm-LC} for the case when $\epsilon_k \equiv 0$, the linear convergence of the descent methods investigated in \cite{AttouchBolte10,AttouchBolte13} for solving the $\ell_p$ regularization problem \eqref{eq-lp}
is presented in the following theorem.

\begin{theorem}\label{thm-LC-DM}
Let $\{x^k\}$ be a sequence satisfying ${\rm (H1)}$ and ${\rm (H2)}$. Then $\{x^k\}$ converges to a critical point $x^*$ of problem \eqref{eq-lp}.
Suppose that $x^*$ is a local minimum of problem \eqref{eq-lp}. Then $\{x^k\}$ converges linearly to $x^*$.
\end{theorem}

\section{Linear convergence of inexact proximal gradient algorithms}
The main purpose of this section is to investigate the linear convergence rate of two inexact PGAs for solving the $\ell_p$ regularization problem \eqref{eq-lp}.
Associated to problem \eqref{eq-PPA}, we denote the (inexact) proximal operator of the $\ell_p$ regularizer by
\begin{equation}\label{eq-PG}
\mathcal{P}_{v,\epsilon}(x):=\epsilon\mbox{-}{\rm arg}\min_{y\in \R^n}\left\{\lambda \|y\|_p^p+\frac{1}{2v}\|y-x\|^2\right\}.
\end{equation}
In the special case when $\epsilon = 0$, we write $\mathcal{P}_{v}(x)$ for $\mathcal{P}_{v,0}(x)$ for simplicity.
Recall that functions $F$ and $H$ are defined by \eqref{eq-Flp}. It is clear that the iterative formula of Algorithm PGA is
\[
x^{k+1}\in \mathcal{P}_{v_k}\left(x^k-v_k\nabla H\left(x^k\right)\right).
\]
Some useful properties of the proximal operator of the $\ell_p$ regularizer are presented as follows.
\begin{proposition}\label{lem-LB}
Let $v>0$, $\epsilon>0$, $x\in \R^n$, $\xi\in \R^n$, $y\in\mathcal{P}_{v}(x-v\nabla H(x))$  and $z\in\mathcal{P}_{v,\epsilon}(x-v(\nabla H(x)+\xi))$. Then the following assertions are true.
\begin{enumerate}[{\rm (i)}]
  \item $F(z)-F(x)\le -\left(\frac{1}{2v}-\|A\|^2\right)\|z-x\|^2-\langle z-x, \xi\rangle+\epsilon$.
  \item For each $i\in \IN$, the following implication holds
    \begin{equation*}\label{eq-supp-1}
    y_i\neq 0 \quad \Rightarrow \quad |y_i|\ge \left(v \lambda p (1-p)\right)^{\frac{1}{2-p}}.
    \end{equation*}
\end{enumerate}
\end{proposition}
\begin{proof}
(i) Recall that $H$ and $\Phi$ are defined by \eqref{eq-Flp}, that is, $H(\cdot)=\|A\cdot-b\|^2$ and $\Phi(\cdot)=\lambda\|\cdot\|_p^p$. It follows from \eqref{eq-PG} that
\[
\Phi(z)+\frac{1}{2v}\|z-(x-v (\nabla H(x)+\xi))\|^2\le \Phi(x)+\frac{1}{2v}\|v (\nabla H(x)+\xi)\|^2+\epsilon,
\]
that is,
\[
\Phi(z)-\Phi(x)\le -\frac{1}{2v}\|z-x\|^2-\langle z-x, 2A^\top (Ax-b)\rangle-\langle z-x, \xi\rangle+\epsilon.
\]
Combining this with \eqref{eq-Axb}, we prove assertion (i) of this theorem.

(ii) Let  $i\in \IN$ be such that $y_i\neq 0$.
Then, by \eqref{eq-PG} (with $\epsilon=0$), one has that
\begin{equation*}\label{eq-GPA-LB-1}
y_i\in {\rm arg}\min_{t\in \R}\left\{\lambda |t|^p+\frac{1}{2v}\left(t-(x-v\nabla H(x))_i\right)^2\right\}.
\end{equation*}
Thus, using its second-order necessary condition, we obtain that $\lambda p(p-1)|y_i|^{p-2}+\frac1{v}\ge 0$;
consequently, $|y_i|\ge \left(v \lambda p (1-p)\right)^{\frac{1}{2-p}}$.  The proof is complete.
\end{proof}

Inspired by the ideas in the seminal work of Rockafellar \cite{roc76}, we propose the following two types of inexact PGAs.

{\scshape Algorithm IPGA-I}.\,
Given an initial point $x^0\in \R^n$, a sequence of stepsizes $\{v_k\}\subseteq \R_+$ and a sequence of inexact terms $\{\epsilon_k\}\subseteq \R_+$. For each $k\in \IN$, having $x^k$, we determine $x^{k+1}$ by
\begin{equation}\label{eq-PGA-II}
x^{k+1}\in \mathcal{P}_{v_k,\epsilon_k}\left(x^k-v_k\nabla H\left(x^k\right)\right).
\end{equation}

{\scshape Algorithm IPGA-II}.\,
Given an initial point $x^0\in \R^n$, a sequence of stepsizes $\{v_k\}\subseteq \R_+$ and a sequence of inexact terms $\{\epsilon_k\}\subseteq \R_+$. For each $k\in \IN$, having $x^k$, we determine $x^{k+1}$ satisfying
\begin{equation}\label{eq-PGA-I}
{\rm dist} \left(x^{k+1},\mathcal{P}_{v_k}\left(x^k-v_k\nabla H\left(x^k\right)\right)\right)\le \epsilon_k.
\end{equation}

\begin{remark}\label{rem-IDM}
{\rm
(i) Algorithms IPGA-I and IPGA-II adopts two popular inexact schemes in the calculation of proximal operators, respectively: Algorithm IPGA-I (resp., Algorithm IPGA-II) measures the inexact term by the approximation of proximal regularized function value (resp., by the distance of the iterate to the exact proximal operator). The latter type of inexact scheme is commonly considered in theoretical analysis, while the former one is more attractive to implement in practical applications. Recently, Frankel et al. \cite{Frankel-JOTA} proposed an inexact PGA (based on a similar inexact scheme to Algorithm IPGA-II) for solving the general 
problem \eqref{eq-CCO}.

(ii) Neither Algorithms IPGA-I nor IPGA-II satisfies both conditions (H1$^\circ$) and (H2$^\circ$) of the inexact descent method mentioned in section 4. Indeed, if both conditions (H1$^\circ$) and (H2$^\circ$) are satisfied, then Lemma \ref{lem-supp} ensures a consistent property of the support of $\{x^k\}$ to $x^*$(cf. \eqref{eq-lem-supp}), which is impossible for either Algorithms IPGA-I or IPGA-II. In particular, Algorithms IPGA-I only satisfies condition (H1$^\circ$) (shown in the proof of Theorem \ref{thm-global}), while neither (H1$^\circ$) nor (H2$^\circ$) can be shown for Algorithms IPGA-II.
}
\end{remark}

Using Theorem \ref{thm-LC}, the global convergence result of Algorithm IPGA-I is presented in the following theorem.
However, we are not able to prove the global convergence of Algorithm IPGA-II at this moment.
\begin{theorem}\label{thm-global}
Let $\{x^k\}$ be a sequence generated by Algorithm IPGA-I with $\{v_k\}$ satisfying
\begin{equation}\label{eq-stepsize}
0<\underline{v}\le v_k\le \bar{v}<\frac12\|A\|^{-2}\quad \mbox{for each } k\in \IN.
\end{equation}
and $\{\epsilon_k\}$ satisfying \eqref{eq-epi}. Suppose that one of limiting points of $\{x^k\}$, denoted by $x^*$, is a local minimum of problem \eqref{eq-lp}. Then $\{x^k\}$ converges to $x^*$.
\end{theorem}
\begin{proof}
In view of Algorithm IPGA-I (cf. \eqref{eq-PGA-II}) and by Proposition \ref{lem-LB}(i) (with $x^{k+1}$, $x^k$, $v_k$, 0, $\epsilon_k$ in place of $z$, $x$, $v$, $\xi$, $\epsilon$), we obtain that
\[
F(x^{k+1})-F(x^k)\le -\left(\frac{1}{2v_k}-\|A\|^2\right)\|x^{k+1}-x^k\|^2+\epsilon_k
\le -\left(\frac{1}{2\bar{v}}-\|A\|^2\right)\|x^{k+1}-x^k\|^2+\epsilon_k
\]
(by \eqref{eq-stepsize}). Note also by \eqref{eq-stepsize} that $\frac{1}{2\bar{v}}-\|A\|^2>0$.
This shows that $\{x^k\}$ satisfies (H1$^\circ$) with $\frac{1}{2\bar{v}}-\|A\|^2$ and $\sqrt{\epsilon_k}$ in place of $\alpha$ and $\epsilon_k$, respectively.
Then the conclusion directly follows from Theorem \ref{thm-LC}(i). The proof is complete.
\end{proof}

Recall that, for the inexact proximal point algorithm (see, e.g., \cite{roc76,WangLGY16}), the inexact term is assumed to have progressively better accuracy to investigate its convergence rate; specifically, it is assumed that $x^{k+1}\in \mathcal{P}_{v_k,\epsilon_k}(x^k)$ with $\epsilon_k=o(\|x^{k+1}-x^{k}\|^2)$ or that ${\rm dist} \left(x^{k+1},\mathcal{P}_{v_k}(x^k)\right)\le o(\|x^{k+1}-x^{k}\|)$. However, we are not able to prove the linear convergence of the inexact PGAs under this assumption of inexact term yet (due to the nonconvexity of the $\ell_p$ regularized function), and we need some additional assumptions to ensure the linear convergence. Recall that $I={\rm supp}(x^*)$ is defined by \eqref{eq-sI}. Let $\{t_k\}\subseteq \R_+$ and $\{\tau_k\}\subseteq \R_+$.
For Algorithms IPGA-I and IPGA-II, we assume
\begin{eqnarray}
x^{k+1}_I\in \mathcal{P}_{v_k,\epsilon_k}\left(\left(x^k-v_k\nabla H(x^k)\right)_I\right) \quad &\mbox{with}& \quad \epsilon_k\le \tau_k\|x_I^{k+1}-x_I^k\|^2, \label{eq-PGA2-1} \\
x^{k+1}_{I^c}\in \mathcal{P}_{v_k,\epsilon_k}\left(\left(x^k-v_k\nabla H(x^k)\right)_{I^c}\right)\quad &\mbox{with}& \quad \epsilon_k\le \tau_k\|x_{I^c}^{k+1}-x_{I^c}^k\|^2,\label{eq-PGA2-2}
\end{eqnarray}
and
\begin{eqnarray}
{\rm dist} \left(x^{k+1}_{I},\left(\mathcal{P}_{v_k}\left(x^k-v_k\nabla H\left(x^k\right)\right)\right)_{I}\right)&\le& t_k\|x_I^{k+1}-x_I^k\|, \label{eq-PGA-c1} \\
{\rm dist}  \left(x^{k+1}_{I^c},\left(\mathcal{P}_{v_k}\left(x^k-v_k\nabla H\left(x^k\right)\right)\right)_{I^c}\right)&\le& t_k\|x_{I^c}^{k+1}-x_{I^c}^k\|,\label{eq-PGA-c2}
\end{eqnarray}
respectively.
Note that \eqref{eq-PGA2-1}-\eqref{eq-PGA2-2} and \eqref{eq-PGA-c1}-\eqref{eq-PGA-c2} are sufficient conditions for guaranteeing \eqref{eq-PGA-II} with $\epsilon_k=t_k\|x^{k+1}-x^{k}\|$ and  \eqref{eq-PGA-I} with $\epsilon_k=t_k\|x^{k+1}-x^{k}\|$, respectively.
(The implementable strategy of inexact PGAs that conditions \eqref{eq-PGA2-1}-\eqref{eq-PGA2-2} or \eqref{eq-PGA-c1}-\eqref{eq-PGA-c2} are satisfied will be proposed at the end of this section.) Now, we establish the linear convergence of the above two inexact PGAs for solving the $\ell_p$ regularization problem under the additional assumptions, respectively. Recall that $f$, $h$ and $\varphi$ are defined by \eqref{eq-fun-g}.
\begin{theorem}\label{thm-IPGA-I}
Let $\{x^k\}$ be a sequence generated by Algorithm IPGA-II with $\{v_k\}$ satisfying \eqref{eq-stepsize}.
Suppose that $\{x^k\}$ converges to a local minimum $x^*$ of problem \eqref{eq-lp}  and that \eqref{eq-PGA-c1} and \eqref{eq-PGA-c2} are satisfied for each $k\in \IN$ with $\lim_{k\to \infty} t_k=0$. Then $\{x^k\}$ converges linearly to $x^*$.
\end{theorem}
\begin{proof}
Note that $\mathcal{P}_{v_k}\left(x^k-v_k\nabla H\left(x^k\right)\right)$ is closed for each $k\in \IN$. Then, by \eqref{eq-PGA-c1} and \eqref{eq-PGA-c2}, one can choose
\begin{equation}\label{eq-yk}
y^k\in \mathcal{P}_{v_k}\left(x^k-v_k\nabla H\left(x^k\right)\right)
\end{equation}
such that
\begin{equation}\label{eq-PGA-c1n}
\|x^{k+1}_{I}-y^k_I\|\le t_k\|x_I^{k+1}-x_I^k\| \quad \mbox{and}\quad \|x^{k+1}_{I^c}-y^k_{I^c}\|\le t_k\|x_{I^c}^{k+1}-x_{I^c}^k\|
\quad \mbox{for each $k\in \IN$.}
\end{equation}
Noting that $x^*_I\in \R_{\neq}^s$ (cf. \eqref{eq-sI}) and recalling that $f$, $h$ and $\varphi$ are defined by \eqref{eq-fun-g}, there exists $0<\delta<\left(\underline{v} \lambda p (1-p)\right)^{\frac{1}{2-p}}$ such that $\mathbf{B}(x_I^*,\delta)\subseteq \R_{\neq}^s$ and
\begin{equation}\label{eq-PGA-c2m}
\|\nabla \varphi(y)-\nabla \varphi(z)\|\le L_{\varphi}\|y-z\|\quad \mbox{for any } y,z\in \mathbf{B}(x_I^*,\delta).
\end{equation}
By the assumption that $\lim_{k\to \infty}x^k=x^*$ and $I={\rm supp}(x^*)$ (cf. \eqref{eq-sI}), we have by \eqref{eq-PGA-c1n} that $\lim_{k\to \infty}y^k_{I}=x^*_{I}$ and $\lim_{k\to \infty}y^k_{I^c}=x^*_{I^c}=0$. Then there exists $N\in \IN$ such that
\[
\|x^k_{I}-x^*_{I}\|\le \delta,\quad \|y^k_{I}-x^*_{I}\|\le \delta \quad \mbox{and} \quad \|y^k_{I^c}\|\le \delta\quad \mbox{for each } k \ge N.
\]
Consequently, one sees that
\begin{equation}\label{eq-thm-PGA-1}
x^k_{I},\,y^k_I\in \mathbf{B}(x_I^*,\delta) \subseteq \R^s_{\neq} \quad \mbox{and} \quad y^k_{I^c}=0\quad \mbox{for each } k \ge N
\end{equation}
(by Proposition \ref{lem-LB}(ii)), and by \eqref{eq-PGA-c2m} that
\begin{equation}\label{eq-PGA-p1}
\|\nabla \varphi(x_I^{k+1})-\nabla \varphi(y_I^{k})\|\le L_{\varphi}\|x_I^{k+1}-y_I^{k}\|\quad \mbox{for each } k \ge N.
\end{equation}
We first provide an estimate on $\{x^k_{I^c}\}_{k\ge N}$. By the assumption that $\lim_{k\to \infty} t_k=0$, we can assume, without loss of generality, that $t_k<\frac12$ for each $k\ge N$. By \eqref{eq-thm-PGA-1}, we obtain from the second inequality of \eqref{eq-PGA-c1n} that
\[
\|x^{k+1}_{I^c}\|\le t_k \|x^{k+1}_{I^c}-x^{k}_{I^c}\|\le t_k \|x^{k+1}_{I^c}\|+ t_k\|x^{k}_{I^c}\|,
\]
and so,
\begin{equation}\label{eq-thm-PGA-2}
\|x^{k+1}_{I^c}\|\le \frac{t_k}{1-t_k} \|x^{k}_{I^c}\|< 2t_k\|x^{k}_{I^c}\| \quad \mbox{for each } k \ge N.
\end{equation}
Below, we estimate $\{x^k_I\}_{k\ge N}$. To do this, we fix $k\ge N$ and let $\tau$ be a constant such that $0<\tau<\frac{1}{4\bar{v}}-\frac12\|A\|^2$
(recalling \eqref{eq-stepsize}). By \eqref{eq-PGA-c1n} and using the triangle inequality, one has that
\begin{equation}\label{eq-thm-PGA-3a}
\frac12\|x_I^{k+1}-x_I^k\|<(1-t_k)\|x_I^{k+1}-x_I^k\|\le \|y_I^k-x_I^k\|\le (1+t_k)\|x_I^{k+1}-x_I^k\| <\frac32\|x_I^{k+1}-x_I^k\|
\end{equation}
(by $t_k<\frac12$). By \eqref{eq-yk}, \eqref{eq-Flp} and \eqref{eq-fun-g}, we check that $y^k_I\in \mathcal{P}_{v_k}\left(x^k_I-v_k\left(\nabla h(x^k_I)+2A_IA_{I^c}x^k_{I^c}\right)\right)$,
and so, we obtain from Proposition \ref{lem-LB}(i) (with $f$, $h$, $A_I$, $y^k_I$, $x^k_I$, $v_k$, $2A_I^\top A_{I^c}x^k_{I^c}$, 0 in place of $F$, $H$, $A$, $z$, $x$, $v$, $\xi$, $\epsilon$) that
\begin{eqnarray}
f(y^k_I)-f(x^k_I)&\le& -\left(\frac{1}{2v_k}-\|A_I\|^2\right)\|y^k_I-x^k_I\|^2-\langle y^k_I-x^k_I, 2A_I^\top A_{I^c}x^k_{I^c}\rangle\nonumber\\
&\le&-\left(\frac{1}{2v_k}-\|A\|^2\right)\|y^k_I-x^k_I\|^2+\tau\|y^k_I-x^k_I\|^2+\frac1\tau\|A\|^4\|x^k_{I^c}\|^2 \label{eq-thm-PGA-3}\\
&\le& -\frac14\left(\frac{1}{2\bar{v}}-\|A\|^2-\tau\right)\|x_I^{k+1}-x_I^k\|^2+\frac1\tau\|A\|^4\|x^k_{I^c}\|^2 \nonumber
\end{eqnarray}
(by \eqref{eq-stepsize} and \eqref{eq-thm-PGA-3a}).
By the smoothness of $f$ on $\mathbf{B}(x_I^*,\delta)(\subseteq \R_{\neq}^s)$ and \eqref{eq-thm-PGA-1}, there exists $L>0$ such that
\begin{equation}\label{eq-thm-PGA-4}
f(x^{k+1}_I)-f(y^k_I)\le \|\nabla f(y^k_I)\|\|x^{k+1}_I-y^k_I\|+L\|x^{k+1}_I-y^k_I\|^2.
\end{equation}
(by Taylor°Øs formula). The first-order optimality condition of \eqref{eq-yk} says that
\begin{equation}\label{eq-thm-PGA-4a}
\nabla \varphi(y^k_I)+\frac1{v_k}\left(y^k_I-x^k_I+2v_kA_I^\top (Ax^k-b)\right)=0.
\end{equation}
Then we obtain by \eqref{eq-fun-g} that
\[
\nabla f(y^k_I)=2A_I^\top (A_Iy^k_I-b)+\nabla \varphi(y^k_I)=-\left(\frac1{v_k}-2A_I^\top A_I\right)(y^k_I-x^k_I)-2A_I^\top A_{I^c}x^k_{I^c};
\]
consequently,
\begin{eqnarray}
\|\nabla f(y^k_I)\|&\le& \left(\frac1{v_k}-2\|A\|^2\right)\|y^k_I-x^k_I\|+2\|A\|^2\|x^k_{I^c}\|\nonumber \\
&\le& \frac32\left(\frac1{\bar{v}}-2\|A\|^2\right)\|x^{k+1}_I-x^k_I\|+2\|A\|^2\|x^k_{I^c}\|\nonumber
\end{eqnarray}
(due to \eqref{eq-stepsize} and \eqref{eq-thm-PGA-3a}). Combing this with \eqref{eq-thm-PGA-4}, we conclude by the first inequality of \eqref{eq-PGA-c1n} that
\begin{eqnarray}
& & f(x^{k+1}_I)-f(y^k_I)\nonumber\\
& &\le \frac32 \left(\frac1{\bar{v}}-2\|A\|^2\right)t_k\|x^{k+1}_I-x^k_I\|^2+2\|A\|^2t_k\|x^k_{I^c}\|\|x^{k+1}_I-x^k_I\|+Lt_k^2\|x^{k+1}_I-x^k_I\|^2\label{eq-thm-PGA-4an}\\
& &\le  \left(\frac32\left(\frac1{\bar{v}}-2\|A\|^2\right)t_k+t_k^2(L+\tau)\right)\|x^{k+1}_I-x^k_I\|^2+ \frac1\tau \|A\|^4\|x^k_{I^c}\|^2.\nonumber
\end{eqnarray}
Recalling that $\lim_{k\to \infty}t_k=0$, we can assume, without loss of generality, that
\[
\frac32\left(\frac1{\bar{v}}-2\|A\|^2\right)t_k +t_k^2(L+\tau)\le \frac14 \tau \quad \mbox{for each } k\ge N.
\]
This, together with \eqref{eq-thm-PGA-3} and \eqref{eq-thm-PGA-4an}, yields that
\begin{equation}\label{eq-thm-PGA-5}
f(x^{k+1}_I)-f(x^k_I)\le -\frac14\left(\frac{1}{2\bar{v}}-\|A\|^2-2\tau\right)\|x_I^{k+1}-x_I^k\|^2+\frac2\tau\|A\|^4\|x^k_{I^c}\|^2.
\end{equation}
On the other hand, by the smoothness of $f$ on $\mathbf{B}(x_I^*,\delta)$, we obtain by \eqref{eq-thm-PGA-1} and \eqref{eq-fun-g} that
\begin{equation}\label{eq-thm-PGA-5a}
\|\nabla f(x_I^{k+1})\|\le \|\nabla h(x^k_I)+\nabla \varphi(y^k_I)\|+\|\nabla h(x^{k+1}_I)-\nabla h(x^k_I)\|+\|\nabla \varphi(x^{k+1}_I)-\nabla \varphi(y^k_I))\|.
\end{equation}
Note by \eqref{eq-thm-PGA-4a}, \eqref{eq-thm-PGA-3a} and \eqref{eq-stepsize} that
\[
\|\nabla h(x^k_I)+\nabla \varphi(y^k_I)\|=\|\frac1{v_k}(x^k_I-y^k_I)-2A^\top_IA_{I^c}x_{I^c}^k\|
\le\frac3{2\underline{v}}\|x^{k+1}_I-x^k_I\|+2\|A\|^2\|x_{I^c}^k\|,
\]
\[
\|\nabla h(x^{k+1}_I)-\nabla h(x^k_I)\|\le 2\|A\|^2\|x^{k+1}_I-x^k_I\|,
\]
and by \eqref{eq-PGA-p1} and \eqref{eq-PGA-c1n} that
\[
\|\nabla \varphi(x^{k+1}_I)-\nabla \varphi(y^k_I)\|\le L_\varphi \|x^{k+1}_I-y^k_I\|\le L_\varphi t_k\|x^{k+1}_I-x^k_I\|.
\]
Hence, \eqref{eq-thm-PGA-5a} implies that
\begin{equation*}\label{eq-thm-PGA-5new}
\|\nabla f(x_I^{k+1})\|\le \left(\frac3{2\underline{v}}+2\|A\|^2+L_\varphi t_k\right)\|x^{k+1}_I-x^k_I\|+2\|A\|^2\|x_{I^c}^k\|.
\end{equation*}
This and \eqref{eq-thm-PGA-5} show that $\{x^k_I\}_{k\ge N}$ satisfies (H1$^\circ$) and (H2$^\circ$) with $f$, $x^k_I$, $\frac14\left(\frac{1}{2\bar{v}}-\|A\|^2-2\tau\right)$, $\left(\frac3{2\underline{v}}+2\|A\|^2+L_\varphi t_k\right)$ and $\max\left\{\sqrt{\frac2\tau},2\right\}\|A\|^2\|x_{I^c}^k\|$
in place of $F$, $x^k$ $\alpha$, $\beta$ and $\epsilon_k$, respectively. Furthermore, it follows from \eqref{eq-thm-PGA-2} that
$\lim_{k\to \infty} \frac{\|x_{I^c}^{k+1}\|}{\|x_{I^c}^k\|}\le \lim_{k\to \infty} 2t_k=0$.
This verifies \eqref{eq-epsilon} assumed in Theorem \ref{thm-LC}(ii). Therefore, the assumptions of Theorem \ref{thm-LC}(ii) are satisfied, and so it follows that $\{x^k_I\}$ converges linearly to $x^*_I$. Recall from \eqref{eq-thm-PGA-2} that $\{x^k_{I^c}\}$ converges linearly to $x^*_{I^c}$ (=0). Therefore, $\{x^k\}$ converges linearly to $x^*$. The proof is complete.
\end{proof}

\begin{remark}\label{rem-IPGA}
{\rm
Frankel et al. \cite{Frankel-JOTA} considered an inexact PGA similar to Algorithm IPGA-II with the inexact control being given by
\begin{equation*}
\epsilon_k=t_k\,{\rm dist} \left(\mathcal{P}_{v_k}\left(x^k-v_k\nabla H\left(x^k\right)\right),\mathcal{P}_{v_k}\left(x^{k-1}-v_{k-1}\nabla H\left(x^{k-1}\right)\right)\right).
\end{equation*}
However, this inexact control would be not convenient to implement for applications because $\epsilon_k$ is expressed in terms of $\mathcal{P}_{v}(\cdot)$ that is usually expensive to calculate exactly. In Theorem \ref{thm-IPGA-I}, we established the linear convergence of Algorithm IPGA-II with the inexact control being given by \eqref{eq-PGA-c1} and \eqref{eq-PGA-c2}.
Our convergence analysis deviates significantly from that of \cite{Frankel-JOTA}, in which the KL inequality is used as a standard technique.
}
\end{remark}

\begin{theorem}\label{thm-IPGA-II}
Let $\{x^k\}$ be a sequence generated by Algorithm IPGA-I with $\{v_k\}$ satisfying \eqref{eq-stepsize}. Suppose that $\{x^k\}$ converges to a global minimum $x^*$ of problem \eqref{eq-lp}  and that \eqref{eq-PGA2-1} and \eqref{eq-PGA2-2} are satisfied for each $k\in \IN$ with $\lim_{k\to \infty} \tau_k=0$. Then $\{x^k\}$ converges linearly to $x^*$.
\end{theorem}
\begin{proof}
For simplicity, we write $y^k\in \mathcal{P}_{v_k}(x^k-v_k\nabla H(x^k))$ for each $k\in \IN$.
By Proposition \ref{lem-LB}(i) (with $y^{k}$, $x^k$, $v_k$, 0, $0$ in place of $z$, $x$, $v$, $\xi$, $\epsilon$) and by \eqref{eq-stepsize}, one has that
\begin{equation*}
\left(\frac{1}{2\bar{v}}-\|A\|^2\right)\|y^k-x^k\|^2\le F(x^k) - F(y^k)\le F(x^k)-\min_{x\in \R^n} F(x).
\end{equation*}
Then, by the assumption that $\{x^k\}$ converges to a global minimum $x^*$ of $F$, we have that $\{y^k\}$ also converges to this $x^*$.
By Theorem \ref{thm-SOG}, it follows from \eqref{eq-2nd} that $2A_I^{\top}A_I+\nabla^2 \varphi(x^*_I)=\nabla^2 f(x^*_I)\succ 0$.
This, together with $x^*_I\in \R_{\neq}^s$ (cf. \eqref{eq-sI}) and the smoothness of $\varphi$ at $x_I^*$, implies that there exists $0<\delta<\left(\underline{v} \lambda p (1-p)\right)^{\frac{1}{2-p}}$ such that \begin{equation}\label{eq-PGA-7}
\mathbf{B}(x^*_I,\delta)\subseteq \R_{\neq}^s\cap \{y\in \R^s:\nabla^2 \varphi(y)\succ -2A_I^\top A_I\}.
\end{equation}
By the convergence of $\{x^k\}$ and $\{y^k\}$ to $x^*$, there exists $N\in \IN$ such that
\begin{equation}\label{eq-PGA-7a}
x_I^k,y_I^k \in \mathbf{B}(x^*_I,\delta),\quad x_{I^c}^k \in \mathbf{B}(0,\delta) \quad \mbox{and} \quad y_{I^c}^k=0\quad \mbox{for each } k \ge N
\end{equation}
(by Proposition \ref{lem-LB}(ii)). Fix $k\ge N$.
Then, by \eqref{eq-PGA2-2} and \eqref{eq-PG}, we have that
\[
\varphi(x^{k+1}_{I^c})+\frac1{2v_k}\|x^{k+1}_{I^c}-x^k_{I^c}+2v_k A_{I^c}^\top  (Ax^k-b)\|^2\le \epsilon_k+\frac{1}{2v_k}\|-x^k_{I^c}+2v_k A_{I^c}^\top  (Ax^k-b)\|^2.
\]
This implies that
\begin{equation}\label{eq-PGA2-3}
\varphi(x^{k+1}_{I^c})\le \epsilon_k+ \frac1{2v_k}\left(\|x^k_{I^c}\|^2-\|x^k_{I^c}-x^{k+1}_{I^c}\|^2\right)-\langle x^{k+1}_{I^c},2A_{I^c}(Ax^k-b)\rangle.
\end{equation}
Note that $\lim_{k\to \infty} x^k_{I^c} =0$ and $\lim_{k\to \infty} \tau_k =0$. By \eqref{eq-PGA2-3} and \eqref{eq-PGA2-2}, there exists $K>0$ such that
\[
\|x^{k+1}_{I^c}\|_p^p\le K(\|x^{k+1}_{I^c}\|+\|x^k_{I^c}\|).
\]
Then it follows from \eqref{eq-norm} (as $p<1$) that
\[
\left(1-K\|x^{k+1}_{I^c}\|^{1-p}\right)\|x^{k+1}_{I^c}\|^p\le \|x^{k+1}_{I^c}\|_p^p- K\|x^{k+1}_{I^c}\| \le K\|x^k_{I^c}\|.
\]
Since $\lim_{k\to \infty} x^k_{I^c} =0$, we assume, without loss of generality, that $\|x^{k+1}_{I^c}\|\le (2K)^{-\frac1{1-p}}$. Hence,
\begin{equation*}\label{eq-PGA2-3a}
\|x^{k+1}_{I^c}\|^p\le 2K\|x^k_{I^c}\|= \left(2K \|x^k_{I^c}\|^{1-p}\right)\|x^k_{I^c}\|^p.
\end{equation*}
Let $\alpha_k:=\left(2K \|x^k_{I^c}\|^{1-p}\right)^{\frac1p}$.
Then it follows that
\begin{equation}\label{eq-PGA-Ic}
\|x^{k+1}_{I^c}-x^k_{I^c}\|\ge \|x^k_{I^c}\|-\|x^{k+1}_{I^c}\|\ge \frac{1-\alpha_k}{\alpha_k}\|x^{k+1}_{I^c}\|.
\end{equation}
On the other hand, let $f_k:\R^s\to \R$ be an auxiliary function defined by
\begin{equation}\label{eq-fk}
f_k(y):=\varphi(y)+\frac1{2v_k}\|y-\left(x^k_I -2v_kA_I^\top (Ax^k-b)\right)\|^2\quad \mbox{for each } y\in \R^s.
\end{equation}
Obviously, $f_k$ is smooth on $\R_{\neq}^s$ and note by Taylor's formula of $f_k$ at $y_I^k$ that
\begin{equation}\label{eq-fkTay}
f_k(y)= f_k(y_I^k)+ \nabla f_k(y_I^k)(y-y_I^k)+ \frac12 \langle y-y_I^k, \nabla^2 f_k(y_I^k)(y-y_I^k)\rangle +o(\|y-y_I^k\|^2), \forall y\in \R^s.
\end{equation}
By \eqref{eq-fk}, it is clear that $y_I^k \in {\rm arg}\min_{y\in \R^s} f_k(y)$. Its first-order necessary optimality condition says that $\nabla f_k(y_I^k)=0$,
and its second-order derivative is $\nabla^2 f_k(y_I^k) =\nabla^2 \varphi(y_I^k)+\frac1{v_k}\mathbf{I}_s$,
where $\mathbf{I}_s$ denotes the identical matrix in $\R^{s\times s}$. Note by \eqref{eq-PGA-7} and \eqref{eq-PGA-7a} that $\nabla^2 \varphi(y_I^k)\succ -2A_I^\top A_I$. Then
\begin{equation*}\label{eq-IPGA-oc2}
\nabla^2 f_k(y_I^k) \succ \frac1{v_k}\mathbf{I}_s-2A_I^\top A_I \succ \frac1{\bar v}\mathbf{I}_s-2A_I^\top A_I \succ 0
\end{equation*}
(by \eqref{eq-stepsize}).
Hence, letting $\sigma$ be the smallest eigenvalue of $\frac1{\bar v}\mathbf{I}_s-2A_I^\top A_I$,  we obtain by \eqref{eq-fkTay} that
\begin{equation}\label{eq-PGA-7c}
f_k(y)\geq f_k(y_I^k)+\frac{\sigma}2\|y-y_I^k\|^2\quad \mbox{for any } y\in \mathbf{B}(y_I^k,2\delta)
\end{equation}
(otherwise we can select a smaller $\delta$). By \eqref{eq-PGA-7a}, one observes that
\[
\|x^{k+1}_I-y^k_I\|\le \|x^{k+1}_I-x^*_I\|+\|y^k_I-x^*_I\|\le 2\delta,
\]
and so, \eqref{eq-PGA-7c} and \eqref{eq-PGA2-1} imply that
\begin{equation*}\label{eq-PGA-8}
\|x^{k+1}_I-y_I^k\|^2\le \frac2{\sigma}\left(f_k(x^{k+1}_I)- f_k(y_I^k)\right)\le \frac2{\sigma}\tau_k\|x_I^{k+1}-x_I^k\|^2.
\end{equation*}
Note that $y^k\in \mathcal{P}_{v_k}(x^k)$ is arbitrary. This, together with \eqref{eq-PGA-Ic}, shows that $\{x^k\}$ can be seen as a special sequence generated by Algorithm IPGA-II that satisfies \eqref{eq-PGA-c1} and \eqref{eq-PGA-c2} with $\max\{\frac{\alpha_k}{1-\alpha_k},\frac2{\sigma}\tau_k\}$ in place of $t_k$. Since $\lim_{k\to \infty} \tau_k=0$ and $\lim_{k\to \infty}\alpha_k=0$ (by the definition of $\alpha_k$), one has that $\lim_{k\to \infty} \max\{\frac{\alpha_k}{1-\alpha_k},\frac2{\sigma}\tau_k\} =0$, and so, the conclusion directly follows from Theorem \ref{thm-IPGA-I}.
\end{proof}

It is a natural question how to design the inexact PGA that satisfies \eqref{eq-PGA2-1}-\eqref{eq-PGA2-2} or \eqref{eq-PGA-c1}-\eqref{eq-PGA-c2}.
Note that both functions $\|\cdot\|_p^p$ and $\|\cdot-x\|^2$ in the proximal operator are separable (see \eqref{eq-PG}).
We can propose two implementable inexact PGAs, Algorithms IPGA-Ip and IPGA-IIp, which are the parallel versions of Algorithms IPGA-I and IPGA-II, respectively.

{\scshape Algorithm IPGA-Ip}.\,
Given an initial point $x^0\in \R^n$, a sequence of stepsizes $\{v_k\}\subseteq \R_+$ and a sequence of nonnegative scalars $\{\epsilon_k\}\subseteq \R_+$. For each $k\in \IN$, having $x^k$, we determine $x^{k+1}$ by
\begin{equation*}
x^{k+1}_i\in \mathcal{P}_{v_k,\epsilon_k}\left(\left(x^k-v_k\nabla H(x^k)\right)_i\right)  \mbox{ with }  \epsilon_k= \tau_k\|x_i^{k+1}-x_i^k\|^2\quad \mbox{for each }i=1,\dots,n.
\end{equation*}

{\scshape Algorithm IPGA-IIp}.\,
Given an initial point $x^0\in \R^n$, a sequence of stepsizes $\{v_k\}\subseteq \R_+$ and a sequence of nonnegative scalars $\{t_k\}\subseteq \R_+$. For each $k\in \IN$, having $x^k$, we determine $x^{k+1}$ satisfying
\begin{equation*}
{\rm dist} \left(x^{k+1}_i,\left(\mathcal{P}_{v_k}\left(x^k-v_k\nabla H\left(x^k\right)\right)\right)_i\right)\le t_k\|x_i^{k+1}-x_i^k\|\quad \mbox{for each }i=1,\dots,n.
\end{equation*}

It is easy to verify that Algorithms IPGA-Ip and IPGA-IIp satisfy conditions \eqref{eq-PGA2-1}-\eqref{eq-PGA2-2} and \eqref{eq-PGA-c1}-\eqref{eq-PGA-c2} respectively, and so, their linear convergence properties follow directly from Theorems \ref{thm-IPGA-I} and \ref{thm-IPGA-II}.

\section{Extension to infinite dimensional cases}
This section extends the results in preceding sections to the infinite-dimensional Hilbert spaces.
In this section, we adopt the following notations.
Let $\mathcal{H}$ be a Hilbert space, and let $\ell^2$ denote the Hilbert space consisting of all square-summable sequences.
We consider the following $\ell_p$ regularized least squares problem in infinite-dimensional Hilbert spaces
\begin{equation}\label{eq-lp-inf}
  \min_{x\in l^{2}}\; F(x):=\|Ax-b\|^2+\sum_{i=1}^\infty\lambda_i|x_i|^p,
\end{equation}
where $A:\ell^2\to \mathcal{H}$ is a bounded linear operator, and $\lambda:=(\lambda_i)$ is a sequence of weights satisfying
\begin{equation}\label{eq-lambda}
\lambda_i\ge \underline{\lambda}>0\quad \mbox{for each } i\in \IN.
\end{equation}

We start from some useful properties of the (inexact) descent methods and then present the linear convergence of (inexact) descent methods and PGA for solving problem \eqref{eq-lp-inf}.
\begin{proposition}\label{lem-supp-inf}
Let $\{x^k\}\subseteq \ell^2$ be a sequence satisfying ${\rm (H1^{\circ})}$ and ${\rm (H2^{\circ})}$, and $\{\epsilon^k\}$ satisfy \eqref{eq-epi}.
Then there exist $N\in \IN$ and a finite index set $J\subseteq \IN$ such that
\begin{equation}\label{eq-lem-supp-inf}
{\rm supp}(x^k)=J\quad \mbox{for each } k\ge N.
\end{equation}
\end{proposition}
\begin{proof}
Fix $k\in \IN$.
By ${\rm (H1^{\circ})}$, one has that
\[
F(x^{k})\le F(x^{k-1})-\alpha \|x^{k}-x^{k-1}\|^2+\epsilon_{k-1}^2\le F(x^{k-1})+\epsilon_{k-1}^2\le F(x^0)+\sum_{i=0}^\infty \epsilon_i^2<+\infty
\]
(due to \eqref{eq-epi}). Then, it follows from \eqref{eq-norm} and \eqref{eq-lambda} that
\[
\|x^k\|^p \le \|x^k\|_p^p \le \frac1{\underline{\lambda}} \sum_{i=1}^\infty\lambda_i|x_i^k|^p\le \frac1{\underline{\lambda}}F(x^{k})
<+\infty.
\]
Then $\{x^k\}$ is bounded, denoting the upper bound of their norms by $M$. Let
\begin{equation}\label{eq-lem-supp-1}
\tau:=\min\left\{\frac1\beta,\left(\frac{\underline{\lambda} p}{2+2\|A\|^2M+2\|A\|\|b\|}\right)^{1-p}\right\}\,(>0).
\end{equation}
Note by Proposition \ref{lem-supp}(i) that $\lim_{k\to \infty} \|x^{k+1}-x^k\|=0$, which, together with \eqref{eq-epi}, shows that there exists $N\in \IN$ such that
\begin{equation}\label{eq-lem-supp-1b}
\|x^{k+1}-x^k\|\le \tau\quad \mbox{and} \quad \epsilon_k< 1\quad \mbox{for each } k\ge N.
\end{equation}
We claim that the following implication is true for for each $k\ge N$ and $i\in \IN$
\begin{equation}\label{eq-lem-supp-1c}
x^{k}_i\neq 0\quad  \Rightarrow \quad |x^{k}_i|> \tau;
\end{equation}
hence, this, together with \eqref{eq-lem-supp-1b}, implies \eqref{eq-lem-supp-inf}, as desired.

Finally, we complete the proof by showing \eqref{eq-lem-supp-1c}. Fix $k> N$ and $i\in \IN$, and suppose that $x^{k}_i\neq 0$. Then, it follows from \eqref{eq-lambda} and (H2$^\circ$) that
\[
\underline{\lambda} p |x^{k}_i|^{p-1}+2A_i^\top(Ax^k-b)\le \|w^{k}\|\le \beta \|x^{k}-x^{k-1}\|+\epsilon_k< 2
\]
(due to \eqref{eq-lem-supp-1b} and $\tau\le \frac1\beta$ by \eqref{eq-lem-supp-1}). Noting that $\|x^k\|\le M$, we obtain from the above relation that
\[
|x^{k}_i|> \left(\frac{\underline{\lambda} p}{2+2\|A\|^2M+2\|A\|\|b\|}\right)^{1-p}\ge \tau
\]
(by \eqref{eq-lem-supp-1}), which verifies \eqref{eq-lem-supp-1c}, as desired.
\end{proof}

\begin{remark}\label{rem-inf}
{\rm
(i) Problem \eqref{eq-lp-inf} for the $n$-dimensional Euclidean space has an equivalent formula to that of problem \eqref{eq-lp}. Indeed, let
$u_i:=\left(\frac{\lambda_i}{\lambda}\right)^{\frac1p}x_i$ and $K_i:=\left(\frac{\lambda}{\lambda_i}\right)^{\frac1p}A_i$
for $i=1,\dots,n.$
Then, problem \eqref{eq-lp-inf} is reformulated to $\min_{u\in \R^n}\; \|Ku-b\|^2+\lambda\|u\|_p^p$ that is \eqref{eq-lp} with $K$ and $u$ in place of $A$ and $x$.

(ii) It is easy to verify by the similar proofs that Theorem \ref{thm-SOG} and Corollary \ref{coro-iso} are also true for problem \eqref{eq-lp-inf} in the infinite-dimensional Hilbert spaces.
}
\end{remark}

\begin{theorem}\label{thm-DM-inf}
Let $\{x^k\}\subseteq \ell^2$ be a sequence satisfying ${\rm (H1)}$ and ${\rm (H2)}$. Then $\{x^k\}$ converges to a critical point $x^*$ of problem \eqref{eq-lp-inf}.
Suppose that $x^*$ is a local minimum of problem \eqref{eq-lp-inf}. Then $\{x^k\}$ converges linearly to $x^*$.
\end{theorem}
\begin{proof}
By the assumptions, it follows from Proposition \ref{lem-supp-inf} that there exist $N\in \IN$ and a finite index set $J$ such that
\eqref{eq-lem-supp-inf} is satisfied. Let $f_J:\R^{|J|}\to \R$ be a function denoted by
\begin{equation*}\label{eq-Fj}
f_J(y):=\|A_Jy-b\|^2+\sum_{i\in J} \lambda_i|y_i|^p\quad \mbox{for any }y\in \R^{|J|}.
\end{equation*}
By the assumptions and \eqref{eq-lem-supp-inf}, we can check that $\{x^k_J\}_{k\ge N}$ satisfies ${\rm (H1)}$ and ${\rm (H2)}$ with $x^k_J$ and $f_J$ in place of $x^k$ and $F$. Hence, the convergence of $\{x^k_J\}$ to a critical point $x^*_J$ of $f_J$ directly follows Theorem \ref{thm-LC-DM}. Let $x^*_{J_c}=0$. Then, by \eqref{eq-lem-supp-inf}, it follows that $\{x^k\}$ converges to this $x^*$, which is a critical point of problem \eqref{eq-lp-inf}. Furthermore, suppose that $x^*$ is a local minimum of problem \eqref{eq-lp-inf}. Then $x^*_J$ is also a local minimum of $f_J$, and so, the linear convergence of $\{x^k_J\}$ to $x^*_J$ also follows from Theorem \ref{thm-LC-DM}. Then, by \eqref{eq-lem-supp-inf}, we conclude that $\{x^k\}$ converges linearly to this $x^*$.
\end{proof}

\begin{theorem}\label{thm-DM2-inf}
Let $\{x^k\}\subseteq \ell^2$ be a sequence satisfying ${\rm (H1^{\circ})}$ and $\{\epsilon^k\}$ satisfy \eqref{eq-epi}. Suppose one of limiting points of $\{x^k\}$, denoted by $x^*$, is a local minimum of problem \eqref{eq-lp-inf}. Then the following assertions are true.
\begin{enumerate}[\rm (i)]
  \item $\{x^k\}$ converges to $x^*$.
  \item Suppose further that  $\{x^k\}$ satisfies ${\rm (H2^{\circ})}$ and $\{\epsilon^k\}$ satisfies \eqref{eq-epsilon}.
    Then $\{x^k\}$ converges linearly to $x^*$.
\end{enumerate}
\end{theorem}
\begin{proof}
The proofs of assertions (i) and (ii) of this theorem use the lines of analysis similar to that of assertion (i) of Theorem \ref{thm-LC} (recalling from Remark \ref{rem-inf}(ii) that Corollary \ref{coro-iso} is true for the infinite-dimensional cases) and that of Theorem \ref{thm-DM-inf}, respectively. The details are omitted.
\end{proof}

Bredies et al. \cite{Bredies2015} investigated the PGA for solving problem \eqref{eq-lp-inf} in infinite-dimensional Hilbert spaces and proved that the generated sequence converges to a critical point under the following additional assumptions: (a) $\{x\in \ell^2:A^\top Ax=\|A^\top A\|x\}$ is finite dimensional, (b) $\|A^\top A\|$ is not an accumulation point of the eigenvalues of $A^\top A$, (c) $A$ satisfies a finite basis injectivity property, and (d) $p$ is a rational. Dropping these technical assumptions, we prove the global convergence of the PGA only under the common made assumption on stepsizes, which significantly improves \cite[Theorem 5.1]{Bredies2015}, and further establish its linear convergence under a simple additional assumption in the following theorem.
Recall from \cite[Theorem 5.1]{AttouchBolte13} that the sequence $\{x^k\}$ generated by Algorithm PGA satisfies conditions (H1) and (H2) under the assumption \eqref{eq-stepsize}. Hence, as an application of Theorem \ref{thm-DM-inf}, the results in the following theorem directly follow.

\begin{theorem}\label{thm-PGA-inf}
Let $\{x^k\}\subseteq \ell^2$ be a sequence generated by Algorithm PGA with $\{v_k\}$ satisfying \eqref{eq-stepsize}.
Then $\{x^k\}$ converges to a critical point $x^*$ of problem \eqref{eq-lp-inf}. Furthermore, suppose that \ $x^*$ is a local minimum of problem \eqref{eq-lp-inf}. Then $\{x^k\}$ converges linearly to $x^*$.
\end{theorem}

Let $x^*$ be a local minimum of problem \eqref{eq-lp-inf}. It was reported in \cite[Theorem 2.1(i)]{ChenXJ10} that
\[
|x_i^*|\ge \left(\frac{\underline{\lambda}p(1-p)}{2\|A_i\|^2}\right)^{\frac1{2-p}} \quad \mbox{for each } i\in {\rm supp}(x^*).
\]
This indicates that ${\rm supp}(x^*)$ is a finite index set. Then, following the proof lines of Theorems \ref{thm-global}-\ref{thm-IPGA-II}, we can obtain the linear convergence of inexact PGAs for infinite-dimensional Hilbert spaces, which are provided as follows.

\begin{theorem}\label{thm-PGAII-inf}
Let $\{x^k\}\subseteq \ell^2$ be a sequence generated by Algorithm IPGA-I with $\{v_k\}$ satisfying \eqref{eq-stepsize}. Then the following assertions are true.
\begin{enumerate}[{\rm (i)}]
  \item Suppose that \eqref{eq-epi} is satisfied, and that one of limiting points of $\{x^k\}$, denoted by $x^*$, is a local minimum of problem \eqref{eq-lp-inf}. Then $\{x^k\}$ converges to $x^*$.
  \item Suppose that $\{x^k\}$ converges to a global minimum $x^*$ of problem \eqref{eq-lp-inf}  and that \eqref{eq-PGA2-1} and \eqref{eq-PGA2-2} are satisfied for each $k\in \IN$ with $\lim_{k\to \infty} \tau_k=0$. Then $\{x^k\}$ converges linearly to $x^*$.
\end{enumerate}
\end{theorem}

\begin{theorem}\label{thm-PGAI-inf}
Let $\{x^k\}\subseteq \ell^2$ be a sequence generated by Algorithm IPGA-II with $\{v_k\}$ satisfying \eqref{eq-stepsize}. Suppose that $\{x^k\}$ converges to a local minimum $x^*$ of problem \eqref{eq-lp-inf} and that \eqref{eq-PGA-c1} and \eqref{eq-PGA-c2} are satisfied for each $k\in \IN$ with $\lim_{k\to \infty} t_k=0$. Then $\{x^k\}$ converges linearly to $x^*$.
\end{theorem}

\begin{remark}
{\rm
Algorithms IPGA-Ip and IPGA-IIp, the parallel versions of Algorithms IPGA-I and IPGA-II, are implementable for solving problem \eqref{eq-lp-inf} in the infinite-dimensional Hilbert spaces, and the generated sequences share the same linear convergence properties as shown in Theorems \ref{thm-PGAII-inf} and \ref{thm-PGAI-inf}, respectively.
}
\end{remark}


\end{document}